\documentclass[final,onefignum,onetabnum]{siamonline250211}

\allowdisplaybreaks
\usepackage[english]{babel}
\usepackage{tikz}

\usepackage{mathrsfs,amsmath}
\usepackage{mathtools}
\usepackage{mathrsfs,amsmath,amsfonts,amssymb}
\usepackage{mathtools} 
\usepackage{graphicx}
\usepackage{xcolor}
\usepackage{bm}

\usepackage{enumitem}
\usepackage{caption}
\usepackage{subcaption}
\usepackage[mathscr]{euscript}
\usepackage{algorithm}
\usepackage{algorithmic}

\newcommand{\defeq}{\vcentcolon=}
\newcommand{\given}{\mathrel{\vert}}

\newcommand{\bvec}[1]{\mathbf{#1}}
\newcommand{\bvecS}[1]{\ensuremath{\boldsymbol{#1}}}
\newcommand{\mat}[1]{\mathbf{{#1}}}
\newcommand{\R}{\mathbb{R}}
\newcommand{\transp}{\top}

\newcommand{\params}{\bvec{m}} 
\newcommand{\paramsFunc}{m} 
\newcommand{\data}{\bvec{d}} 
\newcommand{\Weight}{\bvec{\Pi}} 
\newcommand{\noise}{\bvec{e}} 
\newcommand{\eps}{\bvecS{\varepsilon}} 
\newcommand{\totnoise}{\bvecS{\eta}} 
\newcommand{\sfwd}{\widehat{\FF}} 
\newcommand{\Ofwd}{\FF_{\bvec{O}}} 
\newcommand{\tOfwd}{\widetilde{\FF}_{\bvec{O}}}  
\newcommand{\lfwd}{\bvec{F}}  
\newcommand{\FF}{\bm{\mathrm{\mathscr{F}}}}
\newcommand{\dparmap}[1]{\params_{\scriptscriptstyle{\mathrm{MAP}}}^{{#1}}}

\newcommand{\fo}{pure-BAE}

\newcommand{\EIG}{\Psi} 
\newcommand{\mfEIG}{\EIG_{\!\scriptscriptstyle{\mathrm{MT}}}} 
\newcommand{\mfEIGb}{\bar{\EIG}_{\!\scriptscriptstyle{\mathrm{MT}}}} 
\newcommand{\mfEIGe}{\EIG^\epsilon_{\!\scriptscriptstyle{\mathrm{MT}}}} 
\newcommand{\Op}{\bvec{O}} 
\newcommand{\sensPrec}{\bvec{P}} 

\newcommand{\nSens}{s} 
 
\newcommand{\nDataAll}{d} 
\newcommand{\nChosen}{k} 
\newcommand{\designSubset}{S}
 
\newcommand{\nparams}{n} 
\newcommand{\paramSpace}{\mathscr{M}} 
 
\newcommand{\designSet}{V} 

\newcommand{\likelihood}{\pi_{\data \vert \params}} 
\newcommand{\pipr}{\pi_{\params}} 
\newcommand{\pipost}[1]{\pi^{#1}_{\params \vert \data}} 
\newcommand{\mpost}[1]{\overline{\params}^{#1}_{\data}} 
\newcommand{\Cpost}[1]{\bvecS{\Gamma}^{#1}_{\params \vert \data}} 
\newcommand{\Cnoise}{\bvecS{\Gamma}_{\!\noise}} 
\newcommand{\Cnoisei}{\bvecS{\Gamma}^i_{\!\noise}} 
 
\newcommand{\Cmm}{\bvecS{\Gamma}_{\params\params}} 
\newcommand{\mpr}{\overline{\params}} 
\newcommand{\Cee}{\bvecS{\Gamma}_{\eps\eps}} 
\newcommand{\Cegm}{\bvecS{\Gamma}_{\eps|\params}} 
\newcommand{\Ctgm}{\bvecS{\Gamma}_{\totnoise|\params}} 
\newcommand{\Cem}{\bvecS{\Gamma}_{\eps\params}} 
\newcommand{\Cme}{\bvecS{\Gamma}_{\params\eps}} 
\newcommand{\Cgm}{\bvecS{\Gamma}_{\FF\params}} 
\newcommand{\Cgg}{\bvecS{\Gamma}_{\FF\FF}} 
\newcommand{\Cfm}{\bvecS{\Gamma}_{\lfwd\params}} 
\newcommand{\Cff}{\bvecS{\Gamma}_{\lfwd\lfwd}} 

\newcommand{\norm}[2][]{\left\Vert #2\right\Vert_{#1}} 
\newcommand{\DKL}[2]{\mathcal{D}_{\text{KL}}\left( {#1} \, \| \, {#2} \right)}
\newcommand{\DKLbig}[2]{\mathcal{D}_{\text{KL}}\big( {#1} \, \| \, {#2} \big)}
\newcommand{\Cov}{\mathrm{Cov}}

\newcommand{\parm}{m}

\DeclareMathOperator*{\argmin}{arg\,min}
\DeclareMathOperator*{\argmax}{arg\,max}
\DeclareMathOperator*{\trace}{tr}

\renewcommand{\epsilon}{\varepsilon}

\newtheorem{rmk}{Remark}[section]
\title{Multi-type Sensor Placement for PDE-based Bayesian Inverse Problems\footnote[0]{\funding{
The work of SM and AA was supported in part by 
U.S. Department of Energy, Office of Advanced
Scientific Computing Research Field Work Proposal Number 23-02526.
The work of KK was partially funded by the Carl Zeiss Stiftung through the project ``Model-Based AI: Physical Models and Deep Learning for Imaging and Cancer Treatment''.The work of RN was partially funded by Royal Society of New Zealand Te Ap\=arangi (Marsden Fund Council) Grant MFP-24-UOA-279}}}
\def\addressA{Department of Mathematics, North Carolina State University, Raleigh, NC, USA}
\def\addressB{Interdisciplinary Center for Scientific Computing (IWR),
Heidelberg University, Heidelberg, Germany}
\def\addressC{Department of Engineering Science, University of Auckland,
Auckland, New Zealand}

\author{
Steven Maio\thanks{\addressA}
\and
Alen Alexanderian\footnotemark[1]
\and
Karina Koval\thanks{\addressB}
\and
Ruanui Nicholson\thanks{\addressC}
}

\begin{document}

\maketitle

\begin{abstract}
We address optimal placement of multi-type sensors for Bayesian inverse problems
governed by partial differential equations (PDEs).  The proposed framework
allows for sensors with different accuracies and observation types.  We
formulate the optimal experimental design (OED) problem as a 
knapsack-constrained binary optimization problem
for maximizing expected information gain (EIG). To approximately solve the
resulting optimization problems, we propose a stochastic cost-benefit greedy
algorithm, which admits theoretical guarantees for monotone submodular set
functions.  Specifically, these guarantees apply in the case of linear Gaussian
inverse problems with uncorrelated measurement errors, where the EIG admits a
convenient closed-form expression.
For nonlinear inverse problems, we develop a non-intrusive approach that uses
the Bayesian approximation error framework to define an observation model with
an error-corrected global linear model.  We show that the corresponding
approximate EIG is a lower bound for the exact EIG and thus provides a
principled surrogate objective for the OED problem.  The effectiveness of the
proposed methods is demonstrated in two model inverse problems governed by PDEs.
\end{abstract}


\section{Introduction}\label{sec:intro}
We consider infinite-dimensional Bayesian inverse problems
governed by partial differential equations (PDEs), with data collected
from a set of sensors. In such problems, the statistical quality of the estimated
parameters depends strongly on the design of the sensor network, making 
optimal sensor placement crucial. We consider problems in which one has access
to multiple sensor types. In this setting, we develop efficient methods for
optimal placement of multi-type sensors for linear and nonlinear Bayesian
inverse problems.

The search for an optimal sensor placement can be formulated as an optimal
experimental design (OED) problem~\cite{Ucinski2005}.  OED for
infinite-dimensional inverse problems governed by PDEs is challenging. Specifically, 
upon discretization, one obtains a difficult optimization problem
with an expensive-to-evaluate objective that quantifies the uncertainty or 
information gain
about a high-dimensional inversion parameter~\cite{Alexanderian2021}. 
We focus on finding sensor placements that maximize the
expected information gain (EIG), defined as the Kullback--Leibler divergence
from the posterior to the prior.  For such problems, it is common to consider
the case where the sensors have equal deployment costs
and measure the same quantity.  A typical problem formulation is a
cardinality-constrained binary optimization problem that seeks to identify an
optimal subset from a set of candidate sensor locations; see Section~\ref{subsec:EIGBackground}.

However, in some applications, sensor placement is not limited to deploying
identical sensors throughout the domain. For example, in groundwater flow
applications, one may measure hydraulic head or solute concentration. One may
also measure the same quantity using sensors with different levels of accuracy.
Further, 
different types of sensors typically have different deployment costs. 
In such settings, the experimental design problem is not merely about determining
where to take measurements---one must also determine which sensor types to
deploy and what to measure. This is substantially different from a standard
cardinality-constrained sensor placement problem. Instead of selecting a fixed number of
identical sensors, one must select a collection of sensors that satisfies a
budget constraint, while accounting for the cost and marginal information gain
associated with each candidate sensor. 
This results in a challenging 
knapsack-constrained binary optimization problem.

\paragraph{Related work}\label{subsec:litReview}
There is a growing body of literature devoted to Bayesian OED for inverse
problems governed by PDEs; see, e.g., the review~\cite{Alexanderian2021} or
the more recent efforts~\cite{
Aarset2025,
AlexanderianNicholsonPetra2024,
DuongHelinRojoGarcia2023,
KovalNicholson2025,
WuChenGhattas2023,
WuOLearyRoseberryChenEtAl2023}.
In linear Gaussian problems, the EIG admits a closed-form 
expression~\cite{AlexanderianGloorGhattas2016} which, under the assumption of
uncorrelated Gaussian measurement noise, defines a monotone submodular set
function~\cite{AlexanderianMaio2026}. The latter enables an approximation
guarantee for greedy sensor placement with cardinality
constraints~\cite{KrauseGolovin2014,NemhauserWolseyFisher1978,KrauseSinghGuestrin2008};
see Section~\ref{subsec:submodular-optimization} for background concepts
regarding submodular optimization and the greedy method.  Having nonuniform
costs associated with sensors gives rise to knapsack constraints and motivates
cost-benefit greedy rules~\cite{Sviridenko2004,KrauseGuestrin2005}.  

Since greedy methods can still require many marginal-gain evaluations, several
accelerated variants have been developed, including lazy greedy~\cite{Minoux1978}, 
stochastic greedy~\cite{MirzasoleimanBadanidiyuruKarbasiEtAl2014}, and threshold-based 
methods~\cite{BadanidiyuruVondrak2014}. Stochastic greedy methods are
especially relevant here because they significantly reduce the number of function evaluations
while retaining approximation guarantees in expectation.

Finding sensor placements that maximize the EIG in nonlinear inverse problems is
considerably more challenging than the case of linear inverse problems.  In
nonlinear inverse problems, EIG is generally not available in closed-form and is
often approximated using sampling~\cite{HuanMarzouk2013}, surrogate
models~\cite{WuOLearyRoseberryChenEtAl2023}, or Laplace approximations to
the posterior~\cite{WuChenGhattas2023}.  An alternative approach to design of
nonlinear inverse problems was presented in~\cite{KovalNicholson2025} by 
building on the Bayesian approximation error (BAE) 
framework~\cite{KaipioKolehmainen2013} and 
the global normal linear approximations proposed in~\cite{NicholsonPetraVillaEtAl2023}.

There have also been several works that consider the design of multifidelity or
multimodal sensor networks.  These include data-driven multifidelity sensor
selection under cost constraints~\cite{ClarkBruntonKutz2020} for state estimation, Bayesian
experimental design using a 
Bayesian evidential learning approach for comparing well and geophysical
data~\cite{ThibautCompaireLesparreEtAl2022}, and
multifidelity sensor placement for Bayesian state
estimation~\cite{RamonSarnoskiTumuluriEtAl2026}.  These efforts demonstrate the importance of
moving beyond the standard setting of identical sensors with equal deployment
costs. However, to our knowledge, optimal placement of multi-type sensors in
PDE-based Bayesian inverse problems with high-dimensional parameters is largely unexplored.

\paragraph{Contributions}\label{subsec:Contrib}
This work takes foundational steps toward multi-type sensor placement for linear
and nonlinear Bayesian inverse problems governed by PDEs.  We formulate the problem of finding 
an optimal
placement of multi-type sensors as a knapsack-constrained EIG maximization
problem.  The formulation allows sensors with different costs, accuracies, and
observation types. To provide an accessible presentation, most of the
formulations are presented in a discretized (finite-dimensional) setting.
Throughout, we highlight issues that require care in infinite-dimensional
formulations.

The key contributions of this work are as follows.
\begin{itemize}

\item For linear inverse problems (see Section~\ref{sec:linOED}), we introduce
and analyze a stochastic cost-benefit greedy algorithm for approximately solving
the knapsack-constrained EIG maximization problem.  The analysis yields an
approximation guarantee in expectation and applies more broadly to submodular
maximization under a knapsack constraint.  These approximation guarantees hold
in the case of EIG for linear Gaussian inverse problem with uncorrelated
measurement errors.  We also derive an efficient-to-evaluate expression for the
EIG in the linear case.

\item For nonlinear inverse problems (see Section~\ref{sec:nonlinOED}), building
on the developments in~\cite{KovalNicholson2025}, we devise a non-intrusive
computational framework based on a global linear approximation of the
parameter-to-observable map. We call this approximation the \fo{} global linear model. For
Gaussian priors, we prove that (i) the \fo{} linear model coincides with the
prior expectation of the Jacobian of the parameter-to-observable map; and (ii)
the corresponding approximate EIG---the \fo{} EIG---is a lower bound for the
exact EIG. The first result provides further insight into the use of \fo{}
global linear model in deriving approximate measures of posterior uncertainty;
the second one supports the use of the \fo{} EIG as a principled proxy for the
exact EIG within the EIG maximization setting.

\item We present extensive computational results demonstrating the utility of
the proposed methods for finding near optimal multi-type sensor placements in 
a linear source inversion problem and a nonlinear
porous-medium flow problem; see Section~\ref{sec:NumEx}.

\end{itemize}

\section{Background}\label{sec:bckground}
In this section, we introduce notation and review background materials
on optimal sensor placement for Bayesian inverse problems and submodular optimization.

\subsection{Notation}\label{subsec:notation}
Throughout, we work in a finite-dimensional (discretized) setting.  In the cases where the
forward model is given by a PDE, we assume that the problem has been discretized so that
the inversion parameter vector $\params \in \mathbb{R}^{\nparams}$ represents the
discretized version of a functional model input.  Formulation of Bayesian
inverse problems and sensor selection design problems in the Hilbert space setting
can be found in~\cite{Stuart2010,Alexanderian2026}. 

We note that when discretizing infinite-dimensional inverse problems the
finite-dimensional parameter space is equipped with a suitably weighted inner
product. For example, when using a finite element discretization, a
mass-weighted inner product is used~\cite{Alexanderian2026}.
Working in the correct
inner product is essential to ensure that the discretized quantities 
and operators are consistent with their infinite-dimensional
counterparts.  However, by weighting the operators appropriately, 
one can transform the problem back to the Euclidean space. 
Thus, 
to simplify notation throughout the article, we assume that all
quantities have been transformed so that their matrix representations are with
respect to the Euclidean inner product.
For details on such transformations, we refer to~\cite{AlexanderianSaibaba2018,VillaPetraGhattas2018}.

For a random vector $\bvec{x}$ with the probability density function (PDF)
$\pi$, we denote its mean and covariance matrix by 
$\overline{\bvec{x}} \coloneqq \mathbb{E}_{\pi}[ \bvec{x}]$
and 
$\bvec{\Gamma}_{\bvec{x}\bvec{x}} \coloneqq \Cov_{\pi}(\bvec{x}) = \mathbb{E}_{\pi}\big[ (\bvec{x}-\overline{\bvec{x}})(\bvec{x}-\overline{\bvec{x}})^{\transp} \big]$,
respectively. 
Further, the cross-covariance of random vectors $\bvec{y}$ and $\bvec{z}$, with 
joint PDF $\pi$, is denoted by
\[
\bvec{\Gamma}_{\bvec{y}\bvec{z}} \coloneqq \Cov_{\pi}(\bvec{y},\bvec{z})
= \mathbb{E}_{\pi}\big[
(\bvec{y}-\overline{\bvec{y}})
(\bvec{z}-\overline{\bvec{z}})^{\transp}
\big].
\]

Throughout, with a slight abuse of notation, we denote random variables and
their realizations using the same symbol.  Also, we use $\pi(\bvec{x})$ to
denote the density of a random variable evaluated at a point, and use the same
letter $\pi$ for the joint, marginal, and conditional densities, distinguishing
them with subscripts.  For example, we use  $\pi_{\params}$ for the prior
density, $\pi_{\data}$ for the data marginal, and $\pi_{\params \vert \data}$
for the posterior.  
Additionally, we use the shorthand $\pi =
\mathcal{N}(\overline{\bvec{x}},\bvec{\Gamma}_{\bvec{x}\bvec{x}})$ to mean a
Gaussian density with mean $\overline{\bvec{x}}$ and covariance
$\bvec{\Gamma}_{\bvec{x}\bvec{x}}$. 

Lastly, for a symmetric and positive definite matrix $\bvec{\Gamma}\in\mathbb{R}^{n\times n}$, we define $\|\params\|^2_{\bvec{\Gamma}} = \params^\top \bvec{\Gamma}\params$.

\subsection{Sensor selection for Bayesian inverse problems}\label{subsec:EIGBackground}
We consider the inference of an unknown parameter
$\params \in \mathbb{R}^{\nparams}$ from the observation model 
\begin{equation}\label{eq:model_noS}
    \data = \lfwd(\params) + \noise, \quad \noise \sim \mathcal{N}(\overline{\noise},\Cnoise).
\end{equation}
Here, $\data$ is a data vector, $\noise$ models measurement noise, and  $\lfwd$
is a parameter-to-observable (PTO) map.  In inverse problems governed PDEs,
$\lfwd$ is typically defined as a composition of a PDE solution operator and a
linear observation operator that collects solution data.
Working within a Bayesian paradigm, we assume a Gaussian prior, $\pipr =
\mathcal{N}(\mpr,\Cmm)$, with mean $\mpr \in \mathbb{R}^{\nparams}$ and  
symmetric positive definite covariance matrix $\Cmm$. 

\textbf{Design-dependent observation model}.
We consider inverse problems in which data are collected from a set of sensors,
and focus on the problem of selecting an optimal subset from a
set of candidate sensor locations subject to budget constraints.  To make matters
concrete, here we describe the case where the only constraint is the number of
sensors. This translates to the problem of finding an optimal subset of size $k$
from a set $\designSet = \{1, \ldots, \nSens\}$ that indexes
the candidate sensor locations.
Herein, we
assume each sensor yields a single measurement.
Note that in this case, $\lfwd$
in~\eqref{eq:model_noS} maps $\params$ to a vector of $\nSens$ sensor measurements. 

For $\designSubset
\subset \designSet$, we consider the design-dependent observation model 
\begin{equation}\label{eq:model}
    \data(\designSubset) = \Weight(\designSubset)(\lfwd(\params) + \noise).
\end{equation}
Here, $\Weight(\designSubset) \in \{0,1\}^{|\designSubset| \times \nSens}$ is the
row-selection matrix formed from the rows of the $\nSens \times \nSens$ identity
indexed by $\designSubset$. Thus, $\Weight(\designSubset)$ extracts the
measurements at the selected sensor locations.

The solution to the design-dependent Bayesian inverse problem is the posterior
density
$\pipost{}(\cdot\,; \designSubset) \propto \likelihood(\cdot\,; \designSubset)\pipr$. 
One can show 
$\likelihood(\params; \designSubset) \propto \exp \Big(-\frac{1}{2} \norm[\sensPrec(\designSubset)]{\data - \lfwd(\params) -\overline{\noise}}^2 \Big)$ with
\begin{align}\label{eq:Gamma-S}
    \sensPrec(\designSubset) \defeq
    \Weight(\designSubset)^{\transp}\Big(\Weight(\designSubset)\Cnoise\Weight(\designSubset)^{\transp}\Big)^{-1}\Weight(\designSubset), 
    \quad S \subset V.
\end{align}
In general, exploring this posterior distribution is intractable---it requires 
using expensive sampling procedures. However, in the special case of a linear 
parameter-to-observable (PTO) map, the posterior is Gaussian, $\pipost{}(\designSubset) = \mathcal{N}(\mpost{}(\designSubset),\Cpost{}(\designSubset))$ 
with
{
\setlength{\abovedisplayskip}{2pt}
\setlength{\belowdisplayskip}{2pt}
\begin{align}\label{eq:gaussianPost}
    \mpost{}(\designSubset) &= \Cpost{}(\designSubset)\Big({\lfwd}^{\transp}\Weight(\designSubset)^{\transp}\big(\Weight(\designSubset)\Cnoise\Weight(\designSubset)^{\transp}\big)^{-1}(\data-\overline{\noise})+\Cmm^{-1}\overline{\params}\Big), \\[-4pt]
    \Cpost{}(\designSubset) &= \big({\lfwd}^{\transp}\sensPrec(\designSubset){\lfwd}+ \Cmm^{-1} \big)^{-1}.
\end{align} 
}

\textbf{Optimal sensor placement}.
While there are various OED criteria~\cite{Alexanderian2021,HuanJagalurMarzouk2024}, 
here we focus on choosing a $k$-sensor subset of $V$ that maximizes the expected information gain (EIG),  
\begin{equation}
    \EIG(\designSubset) \!=\! \mathbb{E}_{\data(\designSubset)}\Big[\DKLbig{\pipost{}(\designSubset)}{\pipr}\Big], 
    \;\; \DKLbig{\pipost{}(\designSubset)}{\pipr} \!=\! \mathbb{E}_{\params \vert \data (\designSubset) } \Big[ \log \Big( \frac{\pipost{}(\designSubset)}{\pipr } \Big) \Big].
\end{equation}
The resulting OED problem is a notoriously challenging
optimization problem of the form
\begin{equation}\label{eq:standardOED}
    \max_{\designSubset \subset \designSet} \ \EIG(\designSubset) \quad
    \text{subject to} \  |\designSubset| = k. 
\end{equation}

The first major challenge is the enormous cost of evaluating
$\EIG(\designSubset)$.  In general, one resorts to a double-loop Monte Carlo
sampling approach~\cite{HuanMarzouk2013,Ryan2003}.  A notable exception occurs
when the PTO map is linear, in which case, the EIG admits the following closed-form
expression~\cite{AlexanderianSaibaba2018}:
\begin{equation}
\EIG(\designSubset) = \frac{1}{2}\log \det \big( \bvec{I} + \widetilde{\lfwd}^{\transp}\sensPrec(\designSubset)\widetilde{\lfwd} \big),
\label{eq:EIG_linear}
\end{equation}
where $\widetilde{\lfwd} = \lfwd\Cmm^{\frac{1}{2}}$ is the prior-preconditioned
PTO map.
Despite this explicit expression, optimizing the EIG
for PDE-governed linear inverse problems is still 
challenging due to the high-dimensional parameter space~\cite{AlexanderianSaibaba2018} 
and the need for repeated forward and adjoint solves. 

The second challenge stems from the cardinality constraint on the set
$\designSubset$, which results in an NP-hard combinatorial optimization problem.
To address this issue, various strategies have been proposed, including
relaxation techniques combined with sparsity-inducing
penalties~\cite{HaberHoreshTenorio2008}, continuation
methods~\cite{AlexanderianSaibaba2018,HermanAlexanderianSaibaba2020},
and greedy algorithms for sensor selection~\cite{KrauseGolovin2014,ShamaiahBanerjeeVikalo2010}.

In this work, we focus on greedy algorithms for sensor selection.  Greedy
algorithms are typically non-intrusive and straightforward to implement,
requiring only the ability to evaluate the objective function $\EIG$.
Additionally, many greedy algorithms come with convergence guarantees when the
objective is monotone and submodular~\cite{KrauseGolovin2014}, as we 
detail in the next section.  While the EIG is generally
not submodular, it is under certain assumptions.  Specifically, under our
assumptions of additive Gaussian noise and a Gaussian prior, if we further
assume (i) that the PTO map is linear, and (ii) that the noise covariance
$\Cnoise$ is diagonal, then the
EIG is monotone and
submodular.  

\subsection{Submodular optimization}\label{subsec:submodular-optimization}
Let $\designSet$ be a finite set and $f: 2^{V}\to\mathbb{R}$ a set function.
In our context, $\designSet$ indexes a set of candidate sensors.
We identify a sensor with its index $v\in\designSet$. 

For a subset $S \subset V$ and $v \in V \setminus S$, we define the \textit{marginal
gain} of adding $v$ to $S$ as 
\begin{equation*}\label{eq:def-discrete-derivative}
    \Delta(v | S) \coloneqq f(S \cup \{v\}) - f(S).
\end{equation*}
We say $f$ is \textit{monotone} 
if $\Delta(v | A) \ge 0$ 
for every $A \subset V$ and $v \in V \setminus A$.
Additionally, $f$ is \textit{submodular} if for 
every $A \subset B \subset V$ and $v \in V \setminus B$, we have
$\Delta(v | A) \ge \Delta (v | B)$.
This is a diminishing returns property. In the context of sensor placement, this
says that the utility of adding the sensor $v$ to $A$ is not enhanced by taking the
sensors in $B \setminus A$.

Greedy approaches are common for maximizing monotone submodular functions.  These
algorithms usually begin with some initial set $S_{0}$, a typical choice being
the empty set.  Then, at each iteration $i\ge1$, one chooses a remaining sensor
$v_{i}\in V\setminus S_{i-1} $ that maximizes $\Delta(v | S_{i-1})$ and 
sets $S_i\leftarrow S_{i-1}\cup\{v_i\}$.  This process terminates
when no remaining sensor can be added to $S_{i}$. The greedy solution $S_k$ comes with an
\textit{approximation ratio} of $1 - e^{-1}$ in the case of cardinality
constraints~\cite{KrauseGolovin2014}. Specifically, the greedy solution $S_k$ satisfies
\begin{equation*}
    f(S_k)\ge (1 - e^{-1})\max\{f(S): |S|\le k\}.
\end{equation*}

In the present work, we consider the case where we have multiple sensor types. Thus,
unlike the standard cardinality constrained problem~\eqref{eq:standardOED}, we
need to account for the cost of different sensor types.  Subsequently, we
associate to each sensor $v\in V$ a cost $c(v) > 0$ determined by its sensor type.
We overload this notation and  write 
\[
c(S) \coloneqq \sum_{v \in S}^{}c(v), \quad S \subset V.
\]

The OED problem for the present multi-type sensor setting is
formulated as a knapsack constrained problem of the form
\begin{equation}\label{eq:knapsack-submodular-max-problem}
        \max_{\designSubset\subset\designSet} \ f(S) \quad
        \text{subject to} \ c(S) \le B.
\end{equation}

The standard greedy rule applied to~\eqref{eq:knapsack-submodular-max-problem} 
can perform arbitrarily poorly, because it does not account for the cost of a sensor.
This is addressed via a natural modification of the greedy rule: at
iteration $i$ select $v_{i}$ that maximizes the cost-benefit ratio,
\begin{equation}\label{eq:cost-benefit-rule}
    v_{i} := \argmax_{v \in V\setminus S_{i-1}} \frac{\Delta(v | S_{i-1})}{c(v)}.
\end{equation}
This cost-benefit approach can also perform arbitrarily poorly, because the
cost-benefit ratio hides the total value of a sensor.  One can recover an
approximation ratio by performing both approaches. Namely, we use the standard
greedy approach to obtain a solution $S_1$ and the cost-benefit approach to
obtain $S_2$. We then choose the best of $S_1$ and $S_2$ to ensure an
approximation ratio of $(1 - e^{-1}) / 2$; see~\cite{KrauseGuestrin2005}. 
This procedure is organized in Algorithm~\ref{alg:greedy-knapsack}.  Note that
Algorithm~\ref{alg:greedy-knapsack} requires $\mathcal{O}(|V|^2)$ marginal
gain computations.
This theoretical worst-case cost can often be mitigated by
using lazy evaluations~\cite{Minoux1978}.  Namely, rather than computing the
marginal gains for every remaining sensor, we instead search the remaining
sensors in a specific order based on the previously computed marginal gains.

\begin{algorithm}
\caption{Greedy and cost-benefit greedy selection}
\label{alg:greedy-knapsack}
\begin{algorithmic}[1]
    \REQUIRE{$V$ with $|V|=d$; monotone and submodular $f:2^V\to\mathbb{R}$; cost function $c:V\to \mathbb{R}_+$; budget $B > 0$}
\ENSURE{$S\subset V$ with $c(S)\le B$.}

\STATE $S_1 \leftarrow \emptyset$
\WHILE{$\exists v \in V \setminus S_1$ such that $c(S_1)+c(v)\le B$}
    \STATE $v^* \leftarrow 
    \argmax\big\{
    \Delta(v \mid S_1) :
    v\in V\setminus S_1,\ c(S_1)+c(v)\le B
    \big\}$\
    \STATE $S_1 \leftarrow S_1 \cup \{v^*\}$
\ENDWHILE

\STATE $S_2 \leftarrow \emptyset$
\WHILE{$\exists v \in V \setminus S_2$ such that $c(S_2)+c(v)\le B$}
    \STATE $v^* \leftarrow
    \argmax\big\{
    \frac{\Delta(v \mid S_2)}{c(v)} :
    v\in V\setminus S_2,\ c(S_2)+c(v)\le B
    \big\}$
    \STATE $S_2 \leftarrow S_2 \cup \{v^*\}$
\ENDWHILE

\RETURN $\argmax\{f(S_1),f(S_2)\}$
\end{algorithmic}
\end{algorithm}

\section{Linear inverse problems}\label{sec:linOED}
In this section, we consider multi-type sensor placement  
for linear Bayesian inverse problems. 
The formulation of the corresponding
OED problem is presented in Section~\ref{subsec:formulation}.
We discuss an efficient approach for computing the OED objective 
in Section~\ref{subsec:fastEIG}. 
Subsequently, 
in Section~\ref{subsec:stoch_greedy}, we present our proposed stochastic greedy approach for
solving the optimization problem.

\subsection{Problem formulation}\label{subsec:formulation}
Suppose we have access to $K$ sensor types.
Let sensor type $i \in \{1,\ldots,K\}$ have $\nSens_i$ candidate locations.
For each $i$, the associated observation model is
\begin{equation}\label{eq:Fi}
    \data_i = \lfwd_i\params+\noise_i,\quad\noise_i\sim\mathcal{N}(\overline{\noise}_i,\Cnoisei).
\end{equation}
Here, $\data_i\in\mathbb{R}^{\nSens_i}$ is a vector of sensors
measurements of the $i$th type, $\lfwd_i:\mathbb{R}^{n}\to\mathbb{R}^{\nSens_i}$
is the $i$th parameter-to-observable (PTO) map, and $\noise_i$ is again measurement noise.  
The different sensor types may measure different physical quantities or have 
different levels of accuracy.  Sensor types that measure the same physical
quantity with different levels of accuracy share the same PTO map.  Thus, the
$\lfwd_i$'s in~\eqref{eq:Fi} are not necessarily distinct.  Note that, as
before, we assume each sensor yields a single measurement.  

To formulate the OED problem, we first introduce the \emph{composite} PTO map, 
\begin{equation}\label{eq:fullmap}
    \FF \coloneqq \begin{bmatrix}
        \lfwd_1^\top\;
        \lfwd_2^\top\;
\cdots\;
\lfwd_K^\top
\end{bmatrix}^\top.
\end{equation}
Note that $\FF$ maps the inversion parameter $\params$ to a vector 
in $\R^{\nDataAll}$, where $\nDataAll = \sum_{i=1}^K{\nSens_i}$.
With  the assumption of independent noise across sensor types, we model the full data
vector $\data \in \mathbb{R}^{\nDataAll}$ as
\[
\data = \FF\params + \noise, \quad 
\noise \sim \mathcal{N}(\overline{\noise},\Cnoise),
\]
where $\overline{\noise} = \begin{bmatrix}\overline{\noise}_1^\transp&\cdots&\overline{\noise}_K^\transp\end{bmatrix}^\transp$ and
$\Cnoise = \mathrm{blockdiag}(\Cnoise^1,\cdots,\Cnoise^K)$.

Let $\designSet$ be a set of cardinality $d$ that indexes all candidate sensors. 
As in Section~\ref{subsec:submodular-optimization}, we identify a sensor with its index $v\in\designSet$, i.e., its corresponding row index in $\FF$.
Each $v \in V$ has an associated cost $c(v) > 0$ based on its sensor type.
The design-dependent data model is 
\begin{equation}\label{eq:design_dep_model}
 \data(\designSubset) = \Weight(\designSubset)(\FF \params + \noise),
 \quad 
\designSubset \subset \designSet.
 \end{equation}
Here, $\Weight(S) \in \{0,1\}^{|S| \times d}$
is a selection matrix, which is obtained 
by selecting the rows of the $d \times d$ identity matrix specified by $\designSubset$.

For a prespecified budget $B > 0$, 
the OED problem seeks a design $\designSubset \subset \designSet$ with $c(S)
\leq B$ that maximizes the EIG.
Letting $\widetilde{\FF} =
\FF\Cmm^{{1}/{2}}$, 
the design-dependent EIG for the present multi-type OED problem takes the form 
\begin{equation}\label{eq:mfEIG}
\mfEIG(\designSubset) \defeq \frac{1}{2} \log \det \big( \widetilde{\FF}^{\transp}\sensPrec(\designSubset)\widetilde{\FF} + \bvec{I} \big), 
\quad \designSubset \subset \designSet,
\end{equation}
with $\sensPrec(\designSubset)$ in~\eqref{eq:Gamma-S}. 
Subsequently, we state the multi-type OED problem precisely as follows:
\begin{equation}\label{eq:optim_main}
   \max_{\designSubset \subset \designSet} 
   \ \mfEIG(\designSubset)  \quad 
   \text{subject to} \ \sum_{v \in \designSubset}c(v) \leq B.
\end{equation}
Solving this problem requires overcoming the following 
fundamental challenges: (i) expensive function evaluations; and (ii) a 
knapsack constrained binary optimization problem. The remainder of this 
section is devoted to devising methods that 
address these challenges. 

\subsection{Fast computation of $\mfEIG$}\label{subsec:fastEIG}
Evaluating~\eqref{eq:mfEIG} for a given $\designSubset \subset
\designSet$ requires computing the determinant of an operator on the discretized
parameter space.  In infinite-dimensional inverse problems governed by PDEs, the
dimension $n$ of this space is typically in the order of thousands to millions.
Often, $n$ is significantly larger than the measurement dimension $d$.  In such
cases, it is advantageous to reformulate the EIG criterion so that 
we only need determinants of operators on the measurement space. 
Examples of this approach to single-type sensor placement 
problems with continuous weights can be found in~\cite[Section 4.5]{Alexanderian2021}.

The discrete nature of~\eqref{eq:mfEIG} allows us to further reduce the
dimension from $d$ to the number of selected sensors $k$.  Since $k$ is often
much smaller than $d$, the expression for $\mfEIG$ provided in the following
result offers additional computational savings.  
\begin{proposition}
    The following holds
    \begin{equation}\label{eq:mfEIG_insideout}
    \mfEIG(\designSubset) 
    = \frac{1}{2} \log \Biggl[ \frac{\det \big( \Weight(\designSubset) ( \widetilde{\FF}\widetilde{\FF}^{\transp} + \Cnoise ) \Weight(\designSubset)^{\transp} \big) }{\det \left( \Weight(\designSubset) \Cnoise \Weight(\designSubset)^{\transp} \right)} \Biggr],
    \quad \designSubset\subset\designSet.
    \end{equation}
\end{proposition}

\begin{proof}
We note 
\begin{align}
\mfEIG(\designSubset) &= \frac{1}{2} \log \det \Bigl( \widetilde{\FF}^{\transp}\Weight(\designSubset)^\top\left( \Weight(\designSubset)\Cnoise\Weight(\designSubset)^{\transp}\right)^{-1}\Weight(\designSubset)\widetilde{\FF} + \bvec{I}_{\nparams} \Bigr) \notag\\
& = \frac{1}{2} \log \det \Bigl( \Weight(\designSubset)\widetilde{\FF}\widetilde{\FF}^{\transp}\Weight(\designSubset)^{\transp}\left( \Weight(\designSubset)\Cnoise\Weight(\designSubset)^{\transp}\right)^{-1} + \bvec{I}_{\nChosen} \Bigr) \label{eq:mfEIG_insideout_A}\\
&=  \frac{1}{2} \log \Biggl[ \frac{\det \big( \Weight(\designSubset) ( \widetilde{\FF}\widetilde{\FF}^{\transp} + \Cnoise ) \Weight(\designSubset)^{\transp} \big) }{\det \left( \Weight(\designSubset) \Cnoise \Weight(\designSubset)^{\transp} \right)} \Biggr]. \label{eq:mfEIG_insideout_B} 
\end{align}
The first equality follows from Sylvester's determinant identity; and the
second one
follows by rewriting
$\bvec{I}_{\nChosen}=(\Weight(\designSubset)\Cnoise\Weight(\designSubset)^{\transp})(\Weight(\designSubset)\Cnoise\Weight(\designSubset)^{\transp})^{-1}$.
\end{proof}
\begin{rmk}
A direct log determinant computation via a triangular factorization may be
required when computing~\eqref{eq:mfEIG_insideout_B}
to avoid a potential underflow error. 
\end{rmk}

One can relieve the greedy algorithm from PDE solves by precomputing
$\widetilde{\FF}\widetilde{\FF}^\top$.  This requires an expensive
initial computation.  In the case of a single-type sensor
placement problem, one can also compute a low-rank singular
value decomposition of $\widetilde{\FF}$. We note that this second approach
introduces some approximation errors; see, e.g.,~\cite{AlexanderianSaibaba2018}.
This is suitable for large-scale ill-posed inverse problems where
$\widetilde{\FF}$ exhibits rapid spectral decay. Exploration of such an approach
for the case of multi-type sensor placement problems where $\widetilde{\FF}$ is
the composite forward operator is an interesting line of inquiry for further
work.

An alternative approach to implementing Algorithm~\ref{alg:greedy-knapsack} is
to directly compute the marginal gains rather than re-evaluate the objective.
In the case of a diagonal $\Cnoise$, a direct marginal gain
formula based on the Sherman--Morrison--Woodbury identity can be derived; 
see~\cite{RamonSarnoskiTumuluriEtAl2026}.
This enables fast marginal gain computations at the cost of 
additional memory overhead.  This approach also requires access to the rows of
$\widetilde\FF$.  If $\widetilde\FF$ is too large to store in memory, then an
appropriate caching strategy may reduce the frequency new rows are accessed.

\subsection{Submodular optimization with knapsack constraints}\label{subsec:stoch_greedy}
The deterministic algorithm for knapsack constrained submodular
maximization requires $\mathcal{O}(d^{2})$ marginal gain computations, where $d$
is the number of candidate sensor locations.  When $d$ is large and marginal gain
evaluations are expensive, this cost is prohibitive.  We propose a stochastic
cost-benefit greedy algorithm to reduce this computational cost.  As noted
in the introduction, this approach is motivated by the stochastic variant of the
standard greedy algorithm~\cite{MirzasoleimanBadanidiyuruKarbasiEtAl2014}.  The main idea is to
consider only a subset of the remaining sensors at each iteration. In this
context, an \emph{approximation parameter} $\epsilon$ is used to control the quality of
the approximation. In what follows, we show that this strategy, in expectation, achieves a 
problem dependent approximation ratio.

Assume without a loss of generality that $c(v) \ge 1$ for every $v \in V$.
This can be done via normalization, and does not affect the number of samples per iteration.
We additionally assume that $f(\emptyset)=0$.
Thus, 
$f$ is non-negative due to monotonicity.
We define 
\begin{equation}\label{eq:knapsack-problem-opt-val}
    f^* \defeq \max\{f(S)\,:\,S\subset V, \, c(S)\le B\}.
\end{equation}

The key steps in the proposed stochastic cost-benefit greedy algorithm are as
follows: Begin with the empty set $S_{0} = \emptyset$.  Then, at
iteration $i \ge 1$, we consider a random subset $R_{i} \subset V \backslash
S_{i-1}$ with $T$ sensors and determine a sensor $v_{i}^{*}\in R_{i}$ that
maximizes the weighted marginal gain $\Delta(v | S_{i-1}) / c(v)$ for $v \in R_{i}$.
Subsequently, we set $S_{i} = S_{i-1} \cup \{v^{*}_{i}\}$.
The process is repeated until no remaining sensor can be added.

Let $\epsilon>0$ denote the approximation parameter.
Smaller values of $\epsilon$ produce, in expectation, a better approximation ratio at the cost of additional marginal gain computations.
At iteration $i\ge 1$, let $\widetilde V_{i-1} = \{v \in V\setminus S_{i-1}:c(S_{i-1}) + c(v)\le B\}$.
We then sample $v \in \widetilde V_{i-1}$ with probability $c(v) / c(\widetilde V_{i-1})$.
Our analysis requires
\begin{equation}\label{eq:T-definition}
T = {\frac{c(V)}{B} \log(1/\epsilon)}.
\end{equation}
In practice, $T$ is set to the ceiling of this quantity.
The pseudocode for the algorithm is presented in Algorithm~\ref{alg:stochastic-cb-greedy}.
We can also exploit submodularity by using lazy 
evaluations to reduce the number of marginal gain
computations per inner iteration.
\begin{algorithm}
\caption{Stochastic Cost-Benefit Greedy Algorithm}
\label{alg:stochastic-cb-greedy}
\begin{algorithmic}[1]
\REQUIRE $V$ with $|V|=d$; monotone submodular function
$f:2^V\to\mathbb{R}_+$; cost function $c:V\to\mathbb{R}_+$; budget
$B>0$; approximation parameter $\epsilon>0$
\ENSURE $S\subseteq V$ with $c(S)\le B$

\STATE $S \leftarrow \emptyset$
\STATE $T \leftarrow c(V)\log(1/\epsilon)/B$

\WHILE{$\exists v \in V\setminus S$ such that $c(S)+c(v)\le B$}
    \STATE $R \leftarrow$ a random subset of
    $\{v\in V\setminus S : c(S)+c(v)\le B\}$ with $|R|=T$
    \STATE $v^* \leftarrow
    \argmax\big\{
    \frac{\Delta(v\mid S)}{c(v)} : v\in R
\big\}$\label{alg:best_element}
    \STATE $S \leftarrow S\cup\{v^*\}$
\ENDWHILE

\RETURN $S$
\end{algorithmic}
\end{algorithm}

We begin the analysis of Algorithm~\ref{alg:stochastic-cb-greedy} by considering
the following result that applies to any arbitrary subset $S\subset V$ after one
iteration of the inner loop.
\begin{lemma}[Expected Incremental Gain]\label{lemma:cb-greedy-expected-iterative-gain}
    Let $f$ be monotone and submodular, and satisfy $f(\emptyset)=0$.
Let $S \subset V$ be an arbitrary subset.
Suppose we sample with replacement $T = c(V)\log(1/\epsilon) / B $ sensors from $V \backslash S$, where the probability of sampling $v \in V \backslash S$ is $c(v) / c(V \backslash S)$.
If $v^{*}= \argmax_{v \in R} \Delta(v | S) / c(v)$, then the following estimate holds.
\begin{equation}
    \mathbb{E}\big[f(S\cup\{v^{*}\}) - f(S)\big] = \mathbb{E} \big[\Delta(v^{*}|S)\big]
    \ge \frac{1 - \epsilon}{B}(f^* - f(S))
\end{equation}
\end{lemma}

\begin{proof}
    Let $S^*$ satisfy $f(S^*) = f^*$, where $f^*$ is as in~\eqref{eq:knapsack-problem-opt-val}.
    We denote by $F$ the event that $R \cap (S^{*} \backslash S) \ne \emptyset$. Then,
    \begin{equation}\label{eq:expectation-decomp-estimate}
        \begin{aligned}
            \mathbb{E}\big[\Delta (v^{*}|S)\big]
        &=
        \mathbb{E}\big[\Delta (v^{*}|S) \given R \in F\big] \mathbb{P}(R \in F)
        + \mathbb{E}\big[\Delta (v^{*}|S)\given R \in F^{c}\big] \mathbb{P}(R \in F^{c})\\
        &\ge \mathbb{E}\big[\Delta (v^{*}|S)\given R \in F\big] \mathbb{P}(R \in F).
        \end{aligned}
    \end{equation}

    Suppose that $R \in F$ and consider $v^{*}$ determined by Line~\ref{alg:best_element} of Algorithm~\ref{alg:stochastic-cb-greedy}.
    Next, we define a random variable $A$ with values in $S^{*} \setminus S $ by sampling uniformly $A$ from $R \cap (S^*\setminus S)$.
    That is, if $a \in S^*\setminus S$ appears $j$ times in $R\cap(S^*\setminus S)$, then the probability of $A = a$ is $j / m$ with $m = |R\cap(S^*\setminus S)|$.
    Since $A \in R\cap(S^*\setminus S)$, we then trivially have $\Delta(v^*|S)/c(v)\ge\Delta(A|S)/c(A)$.
    Then by Lemma~\ref{lemma:sample-remaining-element}, $\mathbb{P}(A = a\,|\,R \in F) = c(a) / c(S^{*}\backslash S)$ for $a \in S^{*}\backslash S$.
    Thus, 
    \begin{equation}\label{eq:expectation-intermediate-inequality}
        \begin{aligned}
            \mathbb{E}\big[\Delta (v^{*}|S)\given R \in F\big]
            &\ge \mathbb{E}\big[\Delta (v^{*}|S)/c(v)\given R \in F\big] \ge \mathbb{E}\big[{\Delta(A | S)}/{c(A)} \given R\in F\big]\\
            &= \sum_{a \in S^{*}\backslash S}^{} \frac{\Delta(a | S)}{c(a)} \frac{c(a)}{c(S^{*}\backslash S)}
            = \sum_{a \in S^{*}\backslash S}^{} \frac{\Delta(a | S)}{c(S^{*}\backslash S)}.
        \end{aligned}
    \end{equation}

    Let $S^{*}\backslash S = \{a_{i}\}_{i=1}^{m}$. Then, since $f$ is a monotone submodular function, 
    \begin{multline*}
        f(S^{*}) - f(S)
        \le f(S^{*} \cup S) - f(S)
        = f(S \cup \{a_{i}\}_{i=1}^{m}) - f(S)\\
        = \sum_{j=1}^{m} \big(f(S \cup \{a_{i}\}_{i=1}^{j}) - f(S \cup \{a_{i}\}_{i=1}^{j-1})\big)
        = \sum_{j=1}^{m} \Delta( a_{j} | S \cup \left\{a_{i}\right\}_{i=1}^{j-1})
        \le \sum_{j=1}^{m} \Delta(a_{j} | S).
    \end{multline*}
    Combining this with~\eqref{eq:expectation-intermediate-inequality} yields
    \begin{equation*}
        \mathbb{E}\left[\Delta(v | S)\given R \in F\right] \ge \frac{f(S^{*}) - f(S)}{c(S^{*} \backslash S)}.
    \end{equation*}
    To get the desired inequality, we utilize this estimate and apply Lemma~\ref{lemma:intersection-probability} in~\eqref{eq:expectation-decomp-estimate} to find
    \begin{align*}
        \mathbb{E}\left[\Delta(v^{*}|S)\right]
        & \ge \mathbb{E}\left[\Delta(v^{*}|S) \given R \in F\right] \mathbb{P}(R \in F)\\
        &\ge \frac{f(S^{*}) - f(S)}{c(S^{*}\backslash S)} \cdot \frac{(1-\epsilon)c(S^{*}\backslash S)}{B}
        = (1 - \epsilon) \frac{f(S^{*}) - f(S)}{B}.
    \end{align*}
\end{proof}

We next present the main result of this section, which
states how these expected incremental gains accumulate.
\begin{theorem}[Approximation Ratio]\label{thm:cummulative-approx-ratio}
    Suppose that $f$ is monotone and submodular, and satisfies $f(\emptyset)=0$.
    If Algorithm~\ref{alg:stochastic-cb-greedy} terminates after $k$ iterations, 
    then
    \begin{equation}\label{eq:iteration-approximation-ratio}
        \mathbb{E}\big[f(S_{k})\big] \ge \big\{1 - \exp\big(- {k(1-\epsilon)}/{B}\big)\big\} f^*
    \end{equation}
    In particular, if $\overline C = \max_{v \in V} c(v)$ then
    \begin{equation}\label{eq:approximation-ratio}
        \mathbb{E}\big[f(S_{k})\big] \ge \big\{1 - \exp\big(- {(1-\epsilon)}/{\overline C}\big)\big\}f^*
    \end{equation}
    Moreover, the number of marginal gain computations is $\mathcal{O}(c(V)\log(1/\epsilon))$.
\end{theorem}

\begin{proof}
We first show via induction that
\begin{equation}\label{eq:cummulative-gain-intermediate-inequality}
    \mathbb{E}[f(S_{k})] \ge \Big\{1 - \Big(1 - \frac{1-\epsilon}{B}\Big)^{k}\Big\}f^*.
\end{equation}

For $k=1$, we begin with $S_{0}= \emptyset$. By applying Lemma~\ref{lemma:cb-greedy-expected-iterative-gain}, we get
\begin{equation*}
    \mathbb{E}[f(S_{1})] = \mathbb{E}\big[f(S_{1}) - f(S_{0})\big]
    \ge \frac{1-\epsilon}{B} f(S^{*}) = \Big\{1 - \Big(1 - \frac{1-\epsilon}{B}\Big)\Big\} f^*.
\end{equation*}

Suppose that~\eqref{eq:cummulative-gain-intermediate-inequality} holds for $k\ge 1$.
Note that we can apply Lemma~\ref{lemma:cb-greedy-expected-iterative-gain} if we condition $S_k = S$ for $S\subset V$.
We then have
\begin{equation*}
\begin{aligned}
    \mathbb{E}\big[f(S_{k+1}) - f(S_{k})\big]
    &= \sum_{S \subset V} \mathbb{E}\big[f(S_{k+1}) - f(S_{k}) \,\big|\, S_{k} = S\big]\mathbb{P}(S_{k} = S)\\
    &\ge \sum_{S \subset V} \frac{1-\epsilon}{B}(f^* - f(S)) \mathbb{P}(S_{k} = S)= \Big(\frac{1-\epsilon}{B}\Big)\mathbb{E}\big[f^* - f(S_{k})\big].
\end{aligned}
\end{equation*}
And by applying~\eqref{eq:cummulative-gain-intermediate-inequality}, we find
\begin{align*}
    \mathbb{E}[f(S_{k+1})] 
    &= \mathbb{E}\big[f(S_{k+1}) - f(S_{k})\big] + \mathbb{E}[f(S_{k})]
    \ge \frac{1-\epsilon}{B} f^* + \Big(1 - \frac{1-\epsilon}{B}\Big) \mathbb{E}[f(S_{k})]\\
    &\ge \frac{1-\epsilon}{B} f^* + \Big(1 - \frac{1-\epsilon}{B}\Big)\Big(1 - \Big(1 - \frac{1-\epsilon}{B}\Big)^{k}\Big) f^*
    = \Big\{1 - \Big(1 - \frac{1-\epsilon}{B}\Big)^{k+1}\Big\} f^*.
\end{align*}
Thus, the statement holds for all $k \ge 1$.

The final result follows by an application of the inequality $1 - x \le e^{-x}$.
Moreover, since $k \ge B / \overline C$ we arrive at~\eqref{eq:approximation-ratio}.
\end{proof}

We first note that~\eqref{eq:approximation-ratio} depends on the maximum cost of a sensor.
However, this is a rough estimate based on taking a lower bound on the total number of iterations of Algorithm~\ref{alg:stochastic-cb-greedy}.
In practice, $k$ in~\eqref{eq:iteration-approximation-ratio} may be substantially larger than the lower bound $B / \overline{C}$.
Moreover, in the case of uniform costs, we recover the approximation ratio
from~\cite{MirzasoleimanBadanidiyuruKarbasiEtAl2014}.
Also, by using the Paley--Zygmund inequality~\cite{Petrov2007}, 
we can obtain the following tail lower
bound: 
\begin{corollary}[Tail Lower Bound]
    Suppose $f$ is monotone and submodular, and satisfies $f(\emptyset)=0$.
    If the conditions of Theorem~\ref{thm:cummulative-approx-ratio} hold, then for any $0<\delta<1$,
    \begin{equation*}
        \mathbb{P}\left[f(S_k)\ge (1-\delta)\left(1 - \exp\left(-\frac{1-\epsilon}{\bar C}\right)\right) f^*\right]
        \ge \delta^2
        \left(
            1 - 
            \exp\left(
                -\frac{1-\epsilon}{\bar C}
        \right)
        \right)^2
    \end{equation*}
\end{corollary}

\begin{proof}
    By an application of Theorem~\ref{thm:cummulative-approx-ratio} and Jensen's inequality, we get
    \begin{align*}
        (f^*)^2 \left( 1 - e^{-(1-\epsilon)/\bar C}\right)^2
        \le \mathbb{E}[f(S_k)]^2
        \le \mathbb{E}[ f(S_k)^2]
        \le (f^*)^2.
    \end{align*}
    Also, multiplying~\eqref{eq:approximation-ratio} by $1-\delta$ yields $(1 -\delta) (1 - e^{-(1-\epsilon)/\bar C}) f^* \le (1-\delta) \mathbb{E}[f(S_k)]$.
    Thus,
    \begin{equation*}
        \{
            f(S_k) \ge (1-\delta)\mathbb{E}[f(S_k)]
        \}
        \subset
        \{
            f(S_k) \ge (1-\delta) (1 - e^{-(1-\epsilon)/\bar C}) f^*
        \},
    \end{equation*}
    which implies
    $\mathbb{P}\big[
            f(S_k) \ge (1 - \delta) \mathbb{E}[f(S_k)]
    \big]
    \le
        \mathbb{P}\big[
        f(S_k) \ge (1-\delta)(1 - e^{-(1-\epsilon)/\bar C}) f^*
    \big]$.
    We then apply the Paley--Zygmund inequality to find
    \begin{align*}
        \mathbb{P}\big[
            f(S_k) \ge (1 - \delta) \mathbb{E}[f(S_k)]
    \big]
    \ge \delta^2 \frac{\mathbb{E}[f(S_k)]^2}{\mathbb{E}[f(S_k)^2]}
    \ge \delta^2\left(1 - e^{-(1-\epsilon)/\bar C}\right)^2.
    \end{align*}
\end{proof}

\section{Nonlinear Inverse Problems}\label{sec:nonlinOED}
We now consider problems in which at least one of the parameter-to-observable
(PTO) maps associated with the different sensor types is nonlinear.
In this setting, the composite PTO map
takes the form
\begin{equation}\label{eq:fullmap_nonlinear}
\FF(\params)
:= \begin{bmatrix}
    \lfwd_1(\params)^\top\;
    \lfwd_2(\params)^\top\;
    \cdots \;
    \lfwd_K(\params)^\top
\end{bmatrix}^\top,
\end{equation}
where $\lfwd_i:\mathbb{R}^{n}\to\mathbb{R}^{\nSens_i}$ is the $i$th (potentially nonlinear) PTO map.

As discussed in Section~\ref{subsec:EIGBackground}, the expected information
gain generally does not admit a closed-form expression in this setting, making
the optimal sensor selection problem significantly more challenging.  There are
several sampling methods for approximating the EIG in nonlinear
settings~\cite{HuanJagalurMarzouk2024}. However, using such approaches within an
optimization loop becomes prohibitively expensive. To provide a tractable
approach, here we focus on approaches based on linearization of the PTO maps. We
account for the resulting model error using the Bayesian approximation error
(BAE) approach. 

In particular, following the recent work~\cite{KovalNicholson2025},
we outline a non-intrusive BAE-based framework that uses an observation model 
with a purely data-driven linear model---the \emph{\fo{} global linear model}. 
We elaborate this approach in our setting
in Section~\ref{subsec:nonlin_BAE}. Then, in Section~\ref{subsec:localvsglobal}, we show how the
proposed approach is connected to the commonly used idea of using an ensemble of
local linear models and deriving an approximate expected utility based on the
resulting Laplace approximations~\cite{Alexanderian2021,WuChenGhattas2023}.  
To provide further theoretical justification for the proposed approach, in 
Section~\ref{subsec:EIG_bound} we prove that the
EIG estimate based on the data-driven linear model provides a lower bound for
the exact EIG.

\subsection{Sensor placement via a global error-aware linear
surrogate}\label{subsec:nonlin_BAE}
For now we assume that all $\lfwd_i$'s are nonlinear. In practice, however, some $\lfwd_i$'s may be linear. We discuss this case in Section~\ref{subsec:nonlin_computations}.
We consider a \emph{global} linear surrogate $\mathbb{R}^{d \times n} \ni
\widehat{\FF} \coloneq
[\widehat{\lfwd}_1^{\top} \; \widehat{\lfwd}_2^{\top} \; \cdots \; \widehat{\lfwd}_K^{\top}]^{\top}$ of the full PTO
map $\FF$~\eqref{eq:fullmap_nonlinear}. 
While using such a
linear surrogate for Bayesian inversion is computationally advantageous, it does
entail approximation errors. As established in various
sources~\cite{LunzHauptmannTarvainenEtAl2021,HuttunenKaipio2007},
if the errors and uncertainties introduced by using an approximate forward model
are not accounted for, they can lead to biased reconstructions, overconfident
uncertainty estimates, and, in some cases, entirely misleading inferences.  
Consequently, failure to account for such errors can undermine any OED
procedure, leading to experimental setups that are grossly 
suboptimal for the exact problem~\cite{KovalNicholson2025,AlexanderianNicholsonPetra2024}.
The BAE
approach~\cite{KaipioSomersalo2007,KaipioKolehmainen2013} provides a principled
framework to mitigate these issues by treating the discrepancy between the
accurate and approximate forward models as a random variable, estimating its
statistics, and incorporating it into the likelihood formulation. 

While the BAE
approach is applicable to general, potentially nonlinear, approximate forward
models, we restrict our attention to linear ones. 
This approach begins by restating the relationship between the parameters of
interest and the observations in terms of the approximate model.
In our setting, the design-dependent observation model is
\begin{align}\label{eq: BAEderiv}
\data(\designSubset)=\Weight(\designSubset)\big(\FF(\params)+\noise\big)
=\Weight(\designSubset)\big(\widehat{\FF}\params+\eps(\params)+\noise\big)
=\Weight(\designSubset)\big(\widehat{\FF}\params+\totnoise(\params)\big),
\end{align}
where $\eps(\params):=\FF(\params)-\widehat{\FF}\params$ is termed the {\em
approximation error} and $\totnoise(\params):=\eps(\params)+\noise$ the {\em
total error}. The BAE approach models the joint law of  
$\eps(\params)$ and $\params$
as Gaussian,
\[
\begin{bmatrix} \eps \\ \params \end{bmatrix} \sim \mathcal{N}\left( \begin{bmatrix} \overline{\eps} \\ \overline{\params} \end{bmatrix}, \begin{bmatrix} \Cee & \Cem \\ \Cme & \Cmm \end{bmatrix} \right), 
\]
where $\Cem = \Cme^{\transp}$ denotes the cross-covariance between $\eps$ and
$\params$, and $\overline{\eps}$ and $\Cee$ denote the mean and covariance of
the approximation error $\eps(\params)$ as described in Section~\ref{subsec:notation}.
In the approximate model, 
the law of 
$\eps(\params)$ is replaced by the conditional
law of
$\eps \vert \params$. It is straightforward to show that  
$\eps|\params \sim \mathcal{N}(\overline{\eps}_\params,\Cegm)$,
with \begin{align}\label{eq: BAEstats}
\overline{\eps}_\params=\overline{\eps}+\Cem\Cmm^{-1}(\params-\overline{\params}) 
\quad\text{and}\quad
\Cegm=\Cee-\Cem\Cmm^{-1}\Cme.  
\end{align}

Under the Gaussian assumptions $\noise \sim
\mathcal{N}(\overline{\noise},\Cnoise)$ and $\params \sim
\pi_\params = 
\mathcal{N}(\overline{\params},\Cmm)$ with $\noise\perp\params$, the total
error is also conditionally Gaussian: 
$\totnoise|\params\sim\mathcal{N}(\overline{\totnoise}_\params,\Ctgm)$, with
$\overline{\totnoise}_\params=\overline{\eps}_\params+\overline{\noise}$ and
$\Ctgm=\Cegm+\Cnoise$. 
The error-aware likelihood induced by the approximate
model~\eqref{eq: BAEderiv} is 
\begin{equation}\label{eq:BAE_likelihood}
\likelihood^{\widehat{\FF}}(\designSubset) = \mathcal{N}\left(\Weight(\designSubset)(\widehat{\FF}\params+\overline{\totnoise}_{\params}),\Weight(\designSubset)\Ctgm \Weight(\designSubset)^{\transp}\right), 
\end{equation}
and the corresponding error-aware posterior is $\pipost{\sfwd}(\designSubset)
\propto \likelihood^{\sfwd}(\designSubset)
\pi_\params$, which is also Gaussian. 

The error-aware likelihood was shown to be independent of the choice of $\sfwd$
in the design-free setting~\cite[Theorem 1]{NicholsonPetraVillaEtAl2023}. 
This result was later extended to sensor design~\cite{KovalNicholson2025}; that is,
for any two linear models $\sfwd_1$ and $\sfwd_2$,
\[
\likelihood^{\sfwd_1}(\designSubset)=\likelihood^{\sfwd_2}(\designSubset), \quad 
\text{for all } S \subset V.
\]
Consequently, the posterior
$\pipost{\sfwd}(\designSubset)$ is also invariant to the choice of
linearization. 

An alternative interpretation of this invariance is that, regardless of the
choice of linear operator $\sfwd$, the BAE approach always produces the same
approximate observation model.  To see this, recall that the BAE approach
replaces the accurate model with
\[
\data(\designSubset) = \Weight(\designSubset)\left(\widehat{\FF}\params + \widetilde\totnoise(\params) \right), \quad \params \sim \mathcal{N}(\overline{\params},\Cmm), \quad \widetilde\totnoise(\params) \sim \mathcal{N}(\overline{\totnoise}_{\params},\Ctgm). 
\]
Since the conditional mean $\overline{\totnoise}_{\params} = \overline{\eps} +
\Cem\Cmm^{-1}(\params-\overline{\params})+\overline{\noise}$ is affine in
$\params$, the linear component can be absorbed into the forward operator.
This yields 
\begin{equation}\label{eq:GL_model}
\data(\designSubset) = \Weight(\designSubset)\left(\Ofwd\params+\widetilde{\totnoise}\right), \quad \params \sim \mathcal{N}(\overline{\params},\Cmm), \quad \widetilde{\totnoise} \sim \mathcal{N}(\overline{\totnoise},\Ctgm), \quad \params \perp \tilde{\totnoise}
\end{equation}
where $\Ofwd = \sfwd+\Cem\Cmm^{-1}$ and $\overline{\totnoise} = \overline{\eps}-\Cem\Cmm^{-1}\overline{\params}+\overline{\noise}$.  
Crucially, the choice of $\sfwd$ in the approximate model cancels; cf.\ the proof of~\cite[Theorem 1]{NicholsonPetraVillaEtAl2023}.
Hence, we can restate~\eqref{eq:GL_model} with 
\begin{align}\label{eq:0op}
\Ofwd = \Cgm\Cmm^{-1}, \quad \overline{\totnoise} = \overline{\FF}-\Ofwd\overline{\params}+\overline{\noise}, \quad \Ctgm = \Cnoise+\Cgg-\Cgm\Cmm^{-1}\Cgm^{\transp},
\end{align}
independently of $\sfwd$. We call the operator
$\Ofwd$ in~\eqref{eq:0op} 
the \fo{} global linear model.

The data model~\eqref{eq:GL_model} can equivalently be viewed as the
BAE approximation resulting from using the zero operator $\Op: \params \mapsto
\bvec{0}$ as the surrogate, in which case the approximation error is
$\eps(\params) = \FF(\params)$. Thus, we have 
$\Cem = \Cgm$, $\Cee = \Cgg$,
$\overline{\eps} = \overline{\FF}$, and the error-corrected forward operator
$\Ofwd = \Op+\Cgm\Cmm^{-1}$ recovers the \fo{} global linear model directly. 

The resulting error-aware posterior is $\pipost{\Op}(\designSubset) = \mathcal{N}( \mpost{\Op}(\designSubset), \Cpost{\Op}(\designSubset))$ with
\begin{equation}
    \mpost{\Op}(\designSubset) \!=\! \Cpost{\Op}(\designSubset)\Big(\Ofwd^{\transp}\sensPrec_{\totnoise}(\designSubset)\tilde{\data}+\Cmm^{-1}\overline{\params}\Big)
   \!\quad\!\text{and}\!\quad\!
    \Cpost{\Op}(\designSubset) \!=\! \Big(\Ofwd^{\transp}\sensPrec_{\totnoise}(\designSubset)\Ofwd+ \Cmm^{-1} \Big)^{-1}\!\!.\label{eq:BAEpostcov}
\end{equation}
Here $\sensPrec_{\totnoise}(\designSubset) \!=\!
\Weight(\designSubset)^{\transp}\!\big(\Weight(\designSubset) \Ctgm
\Weight(\designSubset)^{\transp} \big)^{-1}\Weight(\designSubset)$ and
$\tilde{\data} \!=\! \data - \overline{\totnoise}$, where $\data$ is the full-sensor
data. 

Using the multi-type EIG formulation for Gaussian priors and
posteriors~\eqref{eq:mfEIG_insideout}, defining the expected information gain
associated with the error-corrected posterior $\pipost{\Op}(\designSubset)$ is
straightforward. Introducing the prior-preconditioned operator,
$\tOfwd \coloneq
\Ofwd\Cmm^{\frac{1}{2}}$,
the EIG associated with the error-corrected posterior 
$\pipost{\Op}(\designSubset)$ is
\begin{equation}\label{eq:error_corrected_EIG}
\mfEIG^{\Op}(\designSubset) = \frac{1}{2}\log \det \left( {\tOfwd}^{\transp}\sensPrec_{\totnoise}(\designSubset){\tOfwd}+\bvec{I} \right), \quad \designSubset \subset \designSet. 
\end{equation}
We propose using this objective to guide the design selection for the original
nonlinear Bayesian inverse problem. In what follows, we refer to $\mfEIG^{\Op}$
as the \fo{} EIG. Before addressing practical implementation details, we relate
this framework to local linearization methods and establish further theoretical
justification for using~\eqref{eq:error_corrected_EIG} in the OED problem
formulation.

\subsection{Global versus local linearization for optimal sensor selection}
\label{subsec:localvsglobal}

In our proposed approach, we use a global linearization given by the
\fo{} forward operator $\Ofwd=\Cgm \Cmm^{-1}$; cf.~\eqref{eq:GL_model}. 
Although this operator appears to be an algebraic consequence of
using the BAE approach in a particular manner, as shown in this section, under
suitable regularity assumptions and a Gaussian prior, $\Ofwd$ admits a more
direct interpretation: it is the prior expectation of the Jacobian of the
nonlinear PTO map. This provides a connection with Laplace-based approaches to
OED for nonlinear inverse problems, in which the design objective is constructed
using an ensemble of local Gaussian approximations of the
posterior~\cite{Alexanderian2021,GoChen2026,WuChenGhattas2023}.

Since the discussion in this section applies more generally than the multi-type sensor
placement framework under study, for simplicity, we consider a generic 
nonlinear PTO map $\lfwd$.
For large-scale nonlinear inverse problems, a common approach to OED is to consider the
Laplace approximation to the posterior. Namely, given 
a data realization $\data$, we consider the approximation
$\pipost{}(\designSubset)\approx\pi_{\params \vert
\data}^{\rm LA} (\designSubset) =\mathcal{N}(\params_{\rm
MAP}(\data; \designSubset),\mat{\Gamma}_{\rm po}(\designSubset))$.
The maximum a posteriori (MAP) point $\params_{\rm MAP}(\data; \designSubset)$ is given by 
\[
    \dparmap{\data}(\designSubset):=\argmin_{\params\in \paramSpace}-\log{(\pipost{}(\designSubset))}.
\]
Further, the approximate posterior covariance matrix $\mat{\Gamma}_{\rm po}(\designSubset)$ satisfies 
\[
    \mat{\Gamma}_{\rm po}^{-1}(\designSubset):=\mat{H}(\dparmap{\data}(\designSubset)) + \Cmm^{-1},
\]
where 
$\mat{H}(\dparmap{\data}(\designSubset))\defeq \mat{J}(\dparmap{\data}(\designSubset))^{\transp}
\sensPrec(\designSubset)\mat{J}(\dparmap{\data}(\designSubset))$ 
is the Gauss--Newton Hessian, 
with
$\sensPrec(\designSubset)$ as in~\eqref{eq:Gamma-S} and 
$\mat{J}$ the Jacobian of the PTO map, i.e., 
$\mat{J}(\dparmap{}) \defeq
\mathsf{D}_\params\lfwd(\params)\mid_{\params=\dparmap{}}$. 

In this setting, it is common to use an approximate (empirical) expected utility
\begin{align*}
\EIG(\designSubset)=\mathbb{E}_{\data(\designSubset)}\Bigl[\DKL{\pipost{}(\designSubset)}{\pipr}\Bigr]\approx\frac{1}{n_{\rm d}}\sum_{i=1}^{n_{\rm d}}\DKL{\pi_{\params \vert \data_i}^{\rm LA}
(\designSubset)}{\pipr},
\end{align*}
with 
$\data_i = \Weight(\designSubset)(\lfwd(\params_i) + \noise_i)$,
where $\{\params_1,\params_2,\dots,\params_{n_{\rm d}}\}$ are samples from the prior.

When the prior is Gaussian and the Laplace approximation to the posterior is
made, the KL divergence from the approximate posterior to the prior admits  
a closed-form expression. Specifically, 
letting $\mat{H}^i(\designSubset) \defeq \mat{H}(\dparmap{\data_i}(\designSubset))$ 
and $\tilde{\mat{H}}^i \defeq  \Cmm^{\frac{1}{2}}\mat{H}^i(\designSubset) 
\Cmm^{\frac{1}{2}}$, we have 
\[
\DKLbig{\pi_{\params \vert \data_i}^{\rm LA} (\designSubset)}{\pipr} 
=
\frac{1}{2}\Big[ 
    \log\det (\mat{I}+\tilde{\mat{H}}^i(S)) -
\trace\big( \mat{\Gamma}_{\rm po}^{\frac{1}{2}}(\designSubset) \mat{H}^i(\designSubset) \mat{\Gamma}_{\rm po}^{\frac{1}{2}}(\designSubset)\big)
   + 
   \|\dparmap{\data_i}(\designSubset)-\mpr\|_{\Cmm^{-1}}^2\Big].
\]

As noted in~\cite[Section
3.4]{WuChenGhattas2023}, 
one may obtain a further approximation by 
replacing $\{\dparmap{\data_i}(\designSubset)\}_{i=1}^{n_{\mathrm{d}}}$ with
the corresponding prior samples $\{\params_i\}_{i=1}^{n_{\mathrm{d}}}$. 
This 
results in $\norm[\Cmm^{-1}]{\params_i-\mpr}^2$ being independent of the design
and can thus be dropped from the design criterion. This relieves one
from solving multiple MAP estimation problems for each evaluation 
of the  
design criterion;
see~\cite{WuChenGhattas2023} for more details.
The modified design criterion, however, still requires 
evaluation of the Gauss--Newton Hessian at the
prior samples. 
Note also that, for each $i$, the corresponding
Gauss--Newton Hessians can be obtained as the result of (locally)
linearizing the forward model around the prior sample $\params_i$, i.e, using 
$\lfwd(\params)\approx\lfwd(\params_i)+\mat{J}(\params_i)(\params-\params_i)$.

We next show that the \fo{} global linear model coincides with the 
prior expectation of the Jacobian $\mat{J}$. Before formalizing this, we record the following 
technical lemma: 
\begin{lemma}\label{thm:divergTh}
Suppose $\lfwd:\mathbb{R}^\nparams\to\mathbb{R}^\nDataAll$ is differentiable and bounded and $\pi_{\params}=\mathcal{N}(\bvec\parm_0,\Cmm )$. 
Then, we have 
$\int_{\mathbb{R}^\nparams}\mathsf{D}_\params (\lfwd(\bvec\parm)\pi_{\params})\,d\bvec\parm=\mat{0}$.
\end{lemma}
\begin{proof}
Let $B_r(\bvec0)\subset\mathbb{R}^n$ be the open ball of radius $r$ centered at
the origin. By the divergence theorem~\cite[Appendix C]{Evans2022}, 
$\int_{B_r(\bvec0)}\mathsf{D}_\params (\lfwd(\bvec\parm)\pi_{\params})\,d\bvec\parm =
\int_{\partial B_r(\bvec0)}(\lfwd(\bvec\parm)\pi_{\params})\otimes \bvec{n}\,d\bvec s$.
Thus,
$\int_{\mathbb{R}^\nparams}\mathsf{D}_\params (\lfwd(\bvec\parm)\pi_{\params})\;d\bvec\parm=
\lim_{r \to \infty}\int_{\partial B_r(\bvec0)}(\lfwd(\bvec\parm)\pi_{\params})\otimes \bvec{n}\;d\bvec s=\mat{0}$.
The last equality follows from the facts that $\lfwd$ is bounded and 
$\pi_{\params}$ is a multivariate normal PDF.
\end{proof}
The following multidimensional analogue of Stein's lemma~\cite{Stein1981} makes
the interpretation of the \fo{} model as the expectation of the Jacobian $\mat{J}$
precise. 
\begin{theorem}\label{thm:Stein}
    Suppose $\lfwd:\mathbb{R}^\nparams\to\mathbb{R}^\nDataAll$ is differentiable and bounded with Jacobian $\mat{J}(\bvec\parm)\in\mathbb{R}^{\nDataAll\times\nparams}$. Then, letting $\lfwd_0=\mathbb{E}[\lfwd(\bvec\parm)]\in\mathbb{R}^\nDataAll$ be the expected value and 
 $\bvec\Gamma_{\lfwd \params}=\mathbb{E}[(\lfwd(\bvec\parm)-\lfwd_0)(\bvec\parm -\bvec\parm_0)^{\transp}]\in\mathbb{R}^{\nDataAll\times\nparams}$ be the cross covariance matrix, we have
    $\mathbb{E}[\mat{J}(\bvec\parm)]=\bvec\Gamma_{\lfwd \params}\Cmm^{-1}$.
\end{theorem}

\begin{proof}
First we note that for $\pi_{\params}=\mathcal{N}(\bvec\parm_0,\Cmm )$, direct calculation gives 
$\nabla_\params \pi_{\params}=(\mathsf{D}_\params \pi_{\params})^{\transp}=-\Cmm ^{-1}(\bvec\parm-\bvec\parm_0)\pi_{\params}$.
Then, 
letting $\bvec\Gamma_{\lfwd \params}=\mathbb{E}[(\lfwd(\bvec\parm)-\lfwd_0)(\bvec\parm -\bvec\parm_0)^{\transp}]=\mathbb{E}[\lfwd(\bvec\parm)(\bvec\parm -\bvec\parm_0)^{\transp}]$, we have
\begin{align*}
    \mat{0} =\int_{\mathbb{R}^\nparams}\mathsf{D}_\params (\lfwd(\bvec\parm)\pi_{\params})\;d\bvec\parm
    &=\int_{\mathbb{R}^\nparams}\pi_{\params}\mat{J}(\bvec\parm)+\lfwd(\bvec\parm)(\nabla_\params \pi_{\params})^{\transp}\;d\bvec\parm\\
    &=\int_{\mathbb{R}^\nparams}
(\mat{J}(\bvec\parm)-\lfwd(\bvec\parm)(\bvec \parm-\bvec\parm_0)^{\transp}\Cmm ^{-1})\pi_{\params}\;d\bvec\parm\\
&=\int_{\mathbb{R}^\nparams}\mat{J}(\bvec\parm)\pi_{\params}\;d\bvec\parm-\int_{\mathbb{R}^\nparams}\lfwd(\bvec\parm)(\bvec\parm-\bvec\parm_0)^{\transp}\Cmm^{-1}\pi_{\params}\;d\bvec\parm\\
&=\mathbb{E}[\mat{J}(\bvec\parm)]-\mathbb{E}[\lfwd(\bvec\parm)(\bvec\parm-\bvec\parm_0)^T]\Cmm^{-1}\\
&=\mathbb{E}[\mat{J}(\bvec\parm)]-\bvec\Gamma_{\lfwd \params}\Cmm^{-1}. 
\end{align*}
\end{proof}

The preceding discussion shows that the \fo{} global linear
model~\eqref{eq:0op} provides a reasonable alternative to
sensor selection using sample-dependent Laplace approximations.  This enables
replacing Monte Carlo average of KL-divergences involving \emph{local Jacobians}
with a single closed-form expression~\eqref{eq:error_corrected_EIG} built from
the \emph{expected Jacobian}.

\begin{rmk}
The assumptions on $\lfwd$ in Theorem \ref{thm:Stein} can be relaxed. 
This operator does not necessarily need to be bounded. We only need
$\lim_{r \to \infty}\int_{\partial B_r(\bvec0)}(\lfwd(\bvec\parm)\pi_{\params})\otimes \bvec{n}\;d\bvec s=\mat{0}$. 
\end{rmk}

\begin{rmk}
Note that as long as $\lfwd$ is square-integrable with respect to $\pi_{\params}$ the \fo{} operator 
$\bvec\Gamma_{\lfwd \params}\Cmm^{-1}$ is well-defined, 
even if $\lfwd$ is not differentiable.
\end{rmk}

\subsection{A BAE-based lower bound on the EIG}\label{subsec:EIG_bound}
In this section, we prove a lower bound on the exact EIG in terms of the 
\fo{} EIG~\eqref{eq:error_corrected_EIG}.
For readability, we suppress the
design dependency in the model and densities. Since the design enters the
observation model~\eqref{eq:design_dep_model} linearly through the row-selection
matrix $\Weight(\designSubset)$, it can be absorbed into the forward operator
and noise distribution without affecting the structure of the argument. 
Extending the results to the design-dependent setting is thus
straightforward.
Also, as in Section~\ref{subsec:localvsglobal}, we consider a generic nonlinear PTO map $\lfwd$. Thus, throughout this section, $\FF$ and $\Ofwd$ are replaced by $\lfwd$ and $\lfwd_{\bvec{O}}$, respectively, including in equations referenced from Section~\ref{subsec:nonlin_BAE}.

Let $\pi_{\data,\params}$ denote the joint density of $(\data,\params)$ under
the accurate model~\eqref{eq:model} and $\pi^{\Op}_{\data,\params}$
the joint density under the \fo{} global linear model~\eqref{eq:GL_model}. By
construction, 
\[
\pi^{\Op}_{\data,\params} = \mathcal{N}\left( \begin{bmatrix} \overline{\data} \\ \overline{\params} \end{bmatrix}, \begin{bmatrix} \bvec{\Gamma}_{\data\data} & \bvec{\Gamma}_{\data\params} \\ \bvec{\Gamma}_{\params\data} & \Cmm \end{bmatrix} \right).
\]
The following technical lemma shows that the mean and the covariance matrix of 
$\pi^{\Op}_{\data,\params}$ are the same as those of $\pi_{\data,\params}$.
\begin{lemma}\label{lemma:moment_matching}
The following hold,
\[
\begin{aligned}
\mathbb{E}_{\pi^{\Op}_{\data,\params}}\big[(\data,\params)\big] = \mathbb{E}_{\pi_{\data,\params}}\big[ (\data,\params) \big];\quad\text{and} \quad
\Cov_{\pi^{\Op}_{\data,\params}}((\data,\params),(\data,\params)) = \Cov_{\pi_{\data,\params}}((\data,\params),(\data,\params)).
\end{aligned}
\]
\end{lemma}
\begin{proof}
The marginal mean and covariance of $\params$ match by construction. For the remaining quantities, we use the definitions~\eqref{eq:0op}, the independence $\tilde{\totnoise} \perp \params$ in the \fo{} model, and $\noise \perp \params$ in the accurate model. 
We first consider the marginal data mean. We have
\begin{align*}
\overline{\data} = \mathbb{E}_{\pi^{\Op}_{\data,\params}}[\data] &= \lfwd_{\bvec{O}}\overline{\params}+\overline{\totnoise} = \lfwd_{\bvec{O}}\overline{\params}+\overline{\lfwd}-\lfwd_{\bvec{O}}\overline{\params}+\overline{\noise} = \overline{\lfwd}+\overline{\noise} = \mathbb{E}_{\pi_{\data,\params}}[\data].
\end{align*}
For the cross-covariance, independence gives
$\Cov_{\pi^{\Op}}(\widetilde{\totnoise},\params) = \mat{0}$ and
$\Cov_{\pi}(\noise,\params) = \mat{0}$. Thus, 
\begin{align*}
\bvec{\Gamma}_{\data\params} \coloneq \Cov_{\pi^{\Op}}(\data,\params) &= \Cov_{\pi^{\Op}}(\lfwd_{\bvec{O}}\params,\params) =  \lfwd_{\bvec{O}}\Cov_{\pi^{\Op}}(\params,\params) \\
&= \Cfm\Cmm^{-1}\Cmm = \Cfm = \Cov_{\pi}(\data,\params).
\end{align*}
Likewise, utilizing the independence for the data covariance, we have 
\begin{align*}
\bvec{\Gamma}_{\data\data} \coloneq \Cov_{\pi^{\Op}}(\data,\data) &= \Cov_{\pi^{\Op}}(\lfwd_{\bvec{O}}\params,\lfwd_{\bvec{O}}\params)+\Cov_{\pi^{\Op}}(\tilde{\totnoise},\tilde{\totnoise}) \\
&= \lfwd_{\bvec{O}}\Cov_{\pi^{\Op}}(\params,\params)\lfwd_{\bvec{O}}^{\transp}+\Ctgm \\
&= \Cfm\Cmm^{-1}\Cfm^{\transp}+\Cnoise+\Cff-\Cfm\Cmm^{-1}\Cfm^{\transp} \\ 
&= \Cnoise+\Cff = \Cov_{\pi}(\data,\data). 
\end{align*}
\end{proof}
\noindent
The moment-matching property in Lemma~\ref{lemma:moment_matching} plays 
a key role in the following analysis. 

In what follows, we will express the EIG in terms of the \emph{differential entropy}.
Recall that the differential entropy of an $\mathbb{R}^n$-valued 
random variable $\bvec{x}$ with density $\pi$ is 
\[
h(\pi) \coloneqq -\mathbb{E}_{\pi}\big[ \log(\pi(\bvec{x})) \big].
\]  
When $\pi$ is Gaussian, 
$h(\pi) = \frac{1}{2}\left(n\log(2\pi) + \log \det(\Cov_{\pi}(\bvec{x}))+n\right)$; 
see~\cite[Theorem 9.4.1]{Cover1999}. 
The following lemma is another key tool in proving the BAE-based lower bound 
for the EIG.
\begin{lemma}\label{lemma:entropy_matching}
Let $p = \mathcal{N}(\overline{\bvec{x}},\bvec{\Gamma}_{\bvec{x}\bvec{x}})$, and let $\pi$ be a density with $\mathbb{E}_{\pi}[\bvec{x}] = \overline{\bvec{x}}$ and $\Cov_{\pi}(\bvec{x},\bvec{x}) = \bvec{\Gamma}_{\bvec{x}\bvec{x}}$ for $\bvec{x} \in \mathbb{R}^{n}$. Then 
\begin{equation}\label{eq:identities}
-\mathbb{E}_{\pi}\big[ \log(p) \big] = h(p) \quad \text{and} \quad h(p)-h(\pi) = \DKL{\pi}{p}. 
\end{equation}
\end{lemma}
\begin{proof}
Since $p$ is Gaussian, 
\[
-\mathbb{E}_{\pi}\big[\log(p)\big] = \frac{1}{2}\Big(n\log(2\pi)+\log\det(\bvec{\Gamma}_{\bvec{x}\bvec{x}})+\mathbb{E}_{\pi}\big[(\bvec{x}-\overline{\bvec{x}})^{\transp}\bvec{\Gamma}_{\bvec{x}\bvec{x}}^{-1}(\bvec{x}-\overline{\bvec{x}})\big]\Big).
\]
Using the properties of matrix trace and the moment-matching assumption, we have
{
\setlength{\abovedisplayskip}{2pt}
\setlength{\belowdisplayskip}{2pt}
\begin{align*}
\mathbb{E}_{\pi}\big[(\bvec{x}-\overline{\bvec{x}})^{\transp}\bvec{\Gamma}_{\bvec{x}\bvec{x}}^{-1}(\bvec{x}-\overline{\bvec{x}})\big] &\!=\!\mathbb{E}_{\pi}\big[ \trace\big(\bvec{\Gamma}_{\bvec{x}\bvec{x}}^{-1}(\bvec{x}-\overline{\bvec{x}})(\bvec{x}-\overline{\bvec{x}})^{\transp} \big)\big]\\
&\!=\!\trace\Big( \bvec{\Gamma}_{\bvec{x}\bvec{x}}^{-1}\mathbb{E}_{\pi}\big[(\bvec{x}-\overline{\bvec{x}})(\bvec{x}-\overline{\bvec{x}})^{\transp}\big] \Big)
\!=\!\trace\big(\bvec{\Gamma}_{\bvec{x}\bvec{x}}^{-1}\Cov_{\pi}(\bvec{x},\bvec{x}) \big) = \trace[ \bvec{I}_n] \!=\!n.
\end{align*}
}
Therefore, $-\mathbb{E}_{\pi}\big[\log(p)\big] = \frac{1}{2}\big(n\log(2\pi)+\log\det(\bvec{\Gamma}_{\bvec{x}\bvec{x}})+n\big) = h(p)$. 

The second identity in~\eqref{eq:identities} follows directly:
\[
h(p)-h(\pi) = -\mathbb{E}_{\pi}[\log(p)] + \mathbb{E}_{\pi}[\log(\pi)] = \DKL{\pi}{p}. 
\]
\end{proof}

We are now ready to prove the main result of this section.
\begin{theorem}\label{thm:EIG_LB}
Let $\pi_{\data,\params}$ denote the joint density induced by the accurate model~\eqref{eq:model} with Gaussian 
prior $\pi_{\params} = \mathcal{N}(\overline{\params},\Cmm)$ and 
assume that $\pi_{\data,\params}$ has finite differential entropy.
Then, 
\begin{align}\label{eq:EIGbound}
\mfEIG \ge \mfEIG^{\Op},
\end{align}
where
$\mfEIG
:= \mathbb{E}_{\pi_{\data}}\big[ \DKL{\pi_{\params \vert \data}}{\pi_{\params}} \big]$ 
and 
$\mfEIG^{\Op}
:= \mathbb{E}_{\pi^{\Op}_{\data}}\big[ \DKLbig{\pi^{\Op}_{\params \vert \data}}{\pi^{\Op}_{\params}}\big]
$ 
are the exact and \fo{} EIG, respectively.
Furthermore,
$\mfEIG-\mfEIG^{\Op} =\mathbb{E}_{\pi_{\data}}\big[ \DKLbig{\pi_{\params \vert \data}}{\pi^{\Op}_{\params\vert\data}}\big]$.
\end{theorem}
\begin{proof}
We first express the EIG in terms of differential entropies. 
Recalling that $\mfEIG = \mathbb{E}_{\pi_{\data}}\left[ \DKL{\pi_{\params\vert\data}}{\pi_{\params}}\right]$ 
and using the relation $\pi_{\params\vert\data} = \frac{\pi_{\data,\params}}{\pi_{\data}}$, we have 
\begin{align*}
\mfEIG = \mathbb{E}_{\pi_{\data,\params}}\left[ \log \frac{\pi_{\params \vert \data}}{\pi_{\params}}\right] = \mathbb{E}_{\pi_{\data,\params}}\left[\log \frac{\pi_{\data,\params}}{\pi_{\data}\pi_{\params}} \right] 
= -h(\pi_{\data,\params})+h(\pi_{\data})+h(\pi_{\params}).
\end{align*}
An analogous calculation gives $\mfEIG^{\Op} = -h(\pi^{\Op}_{\data,\params})+h(\pi^{\Op}_{\data})+h(\pi^{\Op}_{\params})$. 

Since $\pi_{\params}$ is Gaussian and matches the marginal $\pi^{\Op}_{\params}$, we have $h(\pi^{\Op}_{\params})=h(\pi_{\params})$. Thus, 
\[
\mfEIG-\mfEIG^{\Op}
=  \left(h(\pi^{\Op}_{\data,\params}) - h(\pi_{\data,\params}) \right)-\left( h(\pi^{\Op}_{\data})-h(\pi_{\data}) \right)
\]

Note that by Lemma~\ref{lemma:moment_matching}, $\pi^{\Op}_{\data,\params}$ and
$\pi_{\data,\params}$ have matching mean and covariance. Since marginalization
preserves first and second moments, $\pi^{\Op}_{\data}$ and $\pi_{\data}$ also have
the same mean and covariance matrix. 
Applying
Lemma~\ref{lemma:entropy_matching} gives:
\[
h(\pi^{\Op}_{\data,\params})-h(\pi_{\data,\params})=\DKL{\pi_{\data,\params}}{\pi^{\Op}_{\data,\params}}
\quad\text{and}\quad
h(\pi^{\Op}_{\data})-h(\pi_{\data})=\DKL{\pi_{\data}}{\pi^{\Op}_{\data}}.
\]
Hence,
$\mfEIG-\mfEIG^{\Op}
=\DKLbig{\pi_{\data,\params}}{\pi^{\Op}_{\data,\params}}-\DKLbig{\pi_{\data}}{\pi^{\Op}_{\data}}$.
We next recall that by the 
KL-divergence chain rule~\cite[Page 221]{Murphy2023}, 
\[
\DKL{\pi_{\data,\params}}{\pi^{\Op}_{\data,\params}}
= \DKL{\pi_{\data}}{\pi^{\Op}_{\data}}
+ \mathbb{E}_{\pi_{\data}}\left[ \DKL{\pi_{\params \vert \data}}{\pi^{\Op}_{\params \vert \data} }\right]. 
\]
Therefore, 
$\mfEIG - \mfEIG^{\Op}
=
\mathbb{E}_{\pi_{\data}}\big[ \DKLbig{\pi_{\params \vert \data}
}{\pi^{\Op}_{\params \vert \data} }\big] \ge 0$. 
\end{proof}
 
This lower bound provides insight into when maximizing the \fo{} EIG  can be
effective for obtaining good quality sensor placements. Namely, the gap
$\mathbb{E}_{\pi_{\data}}\big[ \DKLbig{\pi_{\params \vert \data}
}{\pi^{\Op}_{\params \vert \data} }\big]$ measures how well the error-corrected
posterior approximates the true posterior on average over the data distribution.
If this gap does not vary significantly across designs,
then the \fo{} EIG 
roughly preserves the relative ranking of designs and provides a
suitable proxy for $\mfEIG$. Our numerical results in Section~\ref{subsec:subsurface-permeability-inference}
demonstrate the effectiveness of using the \fo{} EIG in the OED problem formulation. 

In the  case $\Cov_{\pi}(\lfwd(\params),\params)
\equiv 0$, the \fo{} linear model captures no (global) linear dependence between data and the parameters. Consequently, the approximate posterior induced by the \fo{} linear model
coincides with the prior (see (\ref{eq:BAEpostcov})), and the associated lower bound on the EIG correctly collapses to zero (see (\ref{eq:EIGbound})). This scenario can be easily detected in practice since $\Cov_{\pi}(\lfwd(\params),\params)$ is estimated when constructing the global
linear model. Also, as shown in~\cite[section
3.1.3]{CuiKovalHerzogEtAl2025}, the statistics of the accurate forward
model that enter the lower bound through $\Ctgm$ (specifically the covariance
$\Cff$) can be used to construct an upper bound on the (exact) EIG.  Comparing
the gap between the two bounds provides a practical diagnostic for the quality
of the \fo{} global linear model at negligible cost.

\subsection{Computational considerations}\label{subsec:nonlin_computations}
In practice, the statistics of the exact PTO map, i.e., ($\Cgm,
\Cgg, \overline{\FF}$), are unknown a priori.  We therefore approximate them
using the standard unbiased Monte Carlo sample-mean, sample-covariance and
sample-cross-covariance estimators using samples $\params^{(i)} \sim
\pipr$ and the corresponding model evaluations $\FF^{(i)} = \FF(\params^{(i)})$
for $i \in \{1, \ldots, N\}$.
Each sample requires one (or more) PDE solve(s) to evaluate
$\FF(\params^{(i)})$. However, these evaluations can be performed in parallel. Once
the statistics have been computed, the assembly and optimization of the
\fo{} EIG~\eqref{eq:error_corrected_EIG} is PDE-free, involving
only linear algebra operations on the precomputed quantities.

A key advantage of the proposed approach is that it is non-intrusive. The precomputation
requires only access to samples $\{(\params^{(i)},\FF(\params^{(i)})\}_{i=1}^N$,
with no need for adjoint solves or Jacobian computations. This makes the method
applicable even in settings where the forward model is a black box, or where the
model itself is unavailable and only input-output data are provided.

We also note a few caveats of the proposed approach.
While the BAE approach results in a Gaussian posterior that is
independent of the choice of linearization $\sfwd$ in the infinite-sample limit,
in the finite-sample setting, the choice of $\sfwd$ can affect the accuracy of
the Monte Carlo estimators; see~\cite{NicholsonPetraVillaEtAl2023}. Choosing a
linearization $\sfwd$ that captures the variability of the nonlinear map $\FF$
can reduce the variance of the approximation error $\eps(\params)$ and hence the
variance of the resulting estimators. A variance reduction strategy
based on Taylor approximations of the forward model has been explored
in~\cite{NicholsonVuchkovVillaEtAl2025}, and can be combined with the approach
presented here to potentially reduce the number of required samples.  One can
also use quasi-Monte Carlo methods to accelerate the convergence of the sample
statistics.

Further, we note that the \fo{} EIG~\eqref{eq:error_corrected_EIG} is not
guaranteed to be submodular due to the non-diagonal structure of $\Ctgm$.  Therefore,
the approximation guarantees in Section~\ref{subsec:stoch_greedy} do not hold.
Nevertheless, the stochastic cost-benefit greedy method performs well in the numerical experiments of
Section~\ref{subsec:subsurface-permeability-inference}. This is consistent with
prior experiences with good performance of greedy sensor placement for general
non-submodular
functions~\cite{AlexanderianNicholsonPetra2024,WuChenGhattas2023,KovalNicholson2025}.
Finally, without submodularity, lazy evaluations within the greedy
procedure are neither justified nor recommended. 

\begin{rmk}\label{rmk:linear_components}
When the composite PTO map $\FF$ contains both nonlinear and linear components,
the \fo{} construction preserves the linear components exactly, applying the BAE
correction only to the nonlinear components. This simplifies both the model
construction and the required precomputations.  In particular, if $\FF(\params)
= [\bvec{F}_{\!\mathrm{n}}(\params)^{\top} \;
(\bvec{F}_{\!\mathrm{l}}\params)^{\top}]^{\top}$, where the nonlinear
($\bvec{F}_{\!\mathrm{n}}$) and linear ($\bvec{F}_\mathrm{l}$) components are
stacked together, then $\Ofwd = [(\bvecS{\Gamma}_{\bvec{F}_{\!\mathrm{n}\params}}
\Cmm^{-1})^{\top} \;  \bvec{F}_{\!\mathrm{l}}^{\top}]^{\top}$, so the linear
components $\bvec{F}_{\!\mathrm{l}}$ enter $\Ofwd$ exactly and do not require
statistical estimation.  Additionally, the mean $\widetilde{\totnoise}$ and
covariance $\Ctgm$ of the total error differ from $\overline{\noise}$ and
$\Cnoise$ only in the blocks corresponding to the nonlinear components.  These
blocks depend solely on the statistics of $\bvec{F}_{\!\mathrm{n}}$, and thus only
those need to be estimated during precomputation. 
\end{rmk}

\section{Numerical Examples}\label{sec:NumEx}
In this section, we present computational results in the context of two model
inverse problems. The first one, considered in
Section~\ref{subsec:heat-source-inversion}, is a linear inverse problem
involving source inversion in a steady advection-diffusion problem. The 
purpose of that study is to examine the performance of the proposed stochastic
cost-benefit greedy method for optimal placement of multi-type sensors in a
setting where the theoretical guarantees established in
Section~\ref{subsec:stoch_greedy} hold.  Then, in
Section~\ref{subsec:subsurface-permeability-inference}, we consider multi-type
sensor placement for a nonlinear inverse problem governed by a porous medium
flow model, where we seek to find optimal placement of pressure and
concentration sensors. The study in that section demonstrates the effectiveness of our
proposed approach for nonlinear inverse problems.

\subsection{Source Inversion in an advection-diffusion problem}\label{subsec:heat-source-inversion}
We consider an inverse problem governed by a steady advection-diffusion problem
to demonstrate the advantages of Algorithm~\ref{alg:stochastic-cb-greedy}.  In
the formulation of this Bayesian inverse problem, we use a
Gaussian prior and an additive Gaussian noise model with uncorrelated
measurement errors.  Under these assumptions, the expected information gain
(EIG), as defined in \eqref{eq:mfEIG}, is a monotone submodular function.
Therefore,  the theoretical guarantees from Section~\ref{subsec:stoch_greedy}
hold.

\subsubsection{Problem setup}
We consider the inverse problem of using sensor data to 
estimate the volume source term $\parm$ 
in the following advection-diffusion equation 
\begin{equation}\label{eq:forward-heat-problem}
\begin{aligned}
    -\nabla\cdot(\kappa \nabla u) + \nabla\cdot( u \bm{v}) = \parm&\quad\text{in}~\Omega,\\
u = 0&\quad\text{on}~\Gamma_{D},\\
-(\kappa\nabla u)\cdot \bm{n} = 0&\quad\mathrm{on}~\Gamma_{N}.
\end{aligned}
\end{equation}
Here, $\Omega = (0,1)^{2}$,
$\Gamma_{D}=[0,1]\times \{0, 1\}$, and $\Gamma_{N} = \{0, 1\} \times[0, 1]$. In the present study, we fix 
$\kappa = 10^{-1}$ and $\bm{v} = [0 \; -1]^\top$.
For the present computational study, we consider a ground-truth inversion 
parameter $m$ that is localized in the center of the domain;
see Figure~\ref{fig:config}~(left).  The corresponding state variable is
shown in the middle panel of the figure. 

The inverse problem uses measurements of $u$ to estimate $\parm$.  
We consider three types of sensors.
The first two sensor types take point measurements of $u$, but with different noise levels.
Sensors of the third type observe
the average value of $u$ over small regions.
The corresponding data models can then be written as
\begin{equation}\label{eq:inverse-source-data-model}
    \data_{i} = \mathcal{B}_{i} u + \noise_{i}, \quad 
    \noise_{i}\sim \mathcal{N}(\mathbf{0},\sigma_{i}^{2}\mathbf{I}),\quad i=1,2,3,
\end{equation}
where $\mathcal{B}_{1}$ and $\mathcal{B}_{2}$ are point evaluation operators, and
 $\mathcal{B}_3$
records the average of $u$ over a ball (within $\Omega$) of a prescribed radius
around the sensor.  

Note that while the first two sensor types are conceptually associated with
pointwise measurements, the corresponding observation operators are still
mathematically defined as localized averaging operators. This ensures
that the observation operators are bounded. When the PDE solution possesses
sufficient regularity, these localized averages may be interpreted as
approximations of exact point evaluations. On the other hand, in the case of
$\mathcal{B}_3$, the average state is computed within a ball around the sensor
whose radius is determined by the properties of the sensing device.
For simplicity, in what follows, we continue to refer to type one and two sensors 
as point evaluation sensors and the third type as point average sensors. 

The point evaluation sensors are placed on a grid of $12\times
12$ equally spaced candidate points in $\Omega$.  We allow the placement
of one or both sensor types in each location.  We set the cost of the low and high quality
sensors to 1 and 2.8, respectively.
Here, we chose $\sigma_{1}^{2}$ and $\sigma_{2}^{2}$ to produce a signal-to-noise ratio (SNR) of 20 and 10, respectively.
The point average sensors are placed on a candidate sensor grid of $6\times 6$
equally spaced points in $\Omega$.  Each of these sensors collects the average value of
$u$ on the intersection of $\Omega$ and a disc of radius 0.20 centered at the
grid point. 
The cost of each point average sensor is 5, and $\sigma_{3}^{2}$ was chosen so
the SNR was 20.
We thus have a data vector 
$\data = [\data_{1}^\top \; \data_{2}^\top \; \data_{3}^\top]^\top\in
\mathbb{R}^{12^{2}+12^{2}+6^{2}}=\mathbb{R}^{324}$.  
The experimental configuration is illustrated in Figure~\ref{fig:config}~(right).
\begin{figure}[t]
    \centering
        {\includegraphics[width=0.30\textwidth]{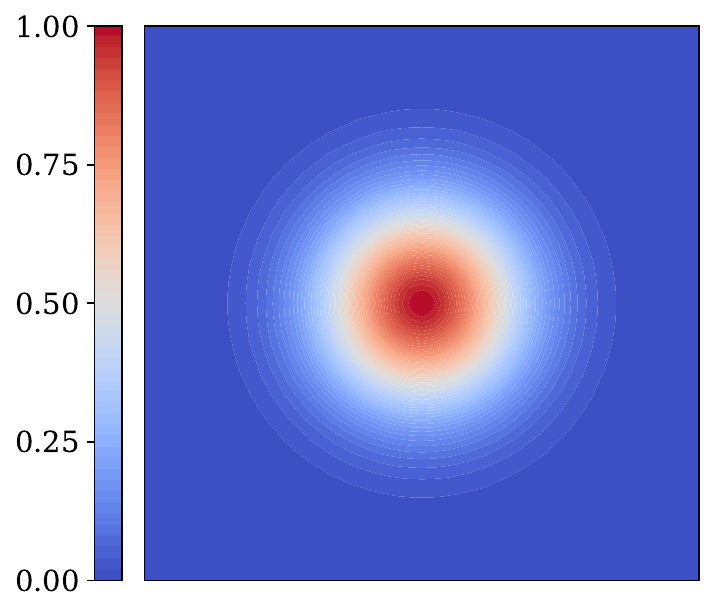}}
        {\includegraphics[width=0.32\textwidth]{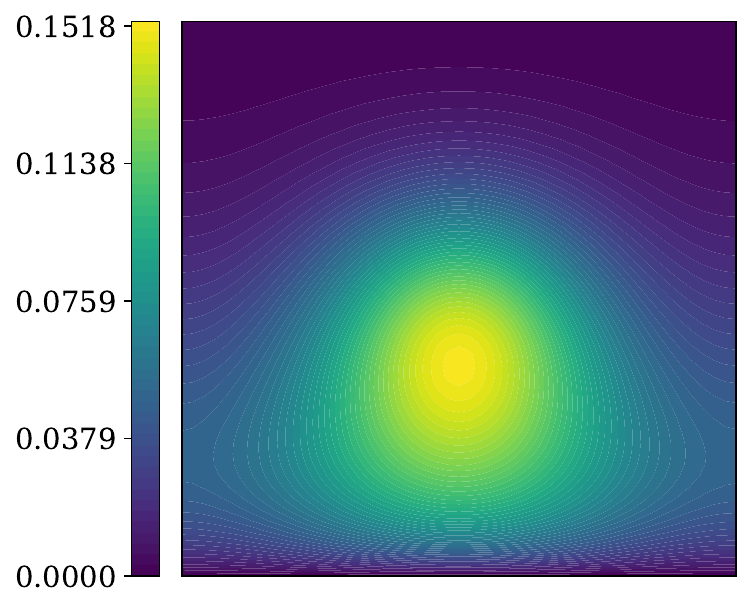}}
    {\includegraphics[width=0.255\textwidth]{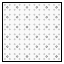}}
    \caption{Left: the ground-truth parameter $m$. Middle: the corresponding 
    state variable. 
    Right: a schematic of the multi-type sensor grid. The candidate locations for 
    point evaluation sensors are indicated by circles and those of point average sensors 
    by diamonds.}
    \label{fig:config}
\end{figure}

We next specify our choice of prior.  
In the infinite-dimensional setting, we use a Gaussian prior measure 
with homogeneous mean and covariance operator given by an
inverse squared elliptic differential operator.
Specifically, the prior covariance operator is defined as $\mathcal{A}^{-2}$
with 
$\mathcal{A}\coloneq
\gamma\Delta + \delta I$. In the present study, we use $\gamma = 0.252$ and $\delta = 0.7$.
We use homogeneous Neumann boundary conditions when computing applications of 
$\mathcal{A}^{-1}$.
For further details on defining such Gaussian priors in a function space setting, 
we refer to~\cite{Stuart2010}.

Finally, the inverse problem is discretized using linear Lagrangian finite elements on a $50\times 50$ uniform triangular mesh.
In total the parameter $\params$ has 2601 degrees of freedom.
The experiments were implemented in Python using
FEniCSx~\cite{BarattaDeanDokkenEtAl2023}.

\subsubsection{Computational Results}
In our numerical tests, we use the (deterministic) best-greedy method 
Algorithm~\ref{alg:greedy-knapsack} as a baseline.  Note that in the present
example, the OED criterion $\mfEIG$ is submodular. Thus, we can use lazy
evaluations to accelerate the computations. 

In our first set of numerical experiments, we consider a fixed budget of $B =
30$ for the knapsack constrained OED problem~\eqref{eq:optim_main}.  To assess the quality of
sensor placements obtained using the proposed stochastic cost-benefit greedy
algorithm, we perform 500 runs of  Algorithm~\ref{alg:stochastic-cb-greedy} with
approximation parameters $\epsilon\in\{0.001, 0.01, 0.1\}$.
Figure~\ref{fig:linear-oed:hist-and-map} compares the performance of these
trials with the best-greedy algorithm.  Let $\mfEIGb$ and $\mfEIGe$ denote the
output of the best-greedy and stochastic cost-benefit greedy methods,
respectively.  The histogram in Figure~\ref{fig:linear-oed:hist-and-map}~(left) 
depicts the relative gap $(\mfEIGb - \mfEIGe) /
\mfEIGb$ over the computed 100 realizations of $\mfEIGe$ and our choices of
$\epsilon$.  Note that smaller values of $\epsilon$ lead to a more concentrated
distribution.  This behavior is expected as smaller $\epsilon$ values result in
an inner iteration closer to that of the deterministic algorithm.  Moreover, the
stochastic approach can also outperform the best-greedy algorithm in some of the
trials. 

To provide further insight, we also consider the MAP points computed using the
computed sensor placements. Figure~\ref{fig:linear-oed:hist-and-map}~(middle) 
shows the MAP point computed using a best-greedy sensor placement.  A
corresponding MAP estimation result is provided using a representative sensor 
placement obtained using the stochastic cost-benefit greedy method. 
\begin{figure}[t!]
\centering
\begin{tikzpicture}
	\node[anchor=center, inner sep=0] (left) at (7.8, 0)
    {\includegraphics[width=0.6\textwidth]{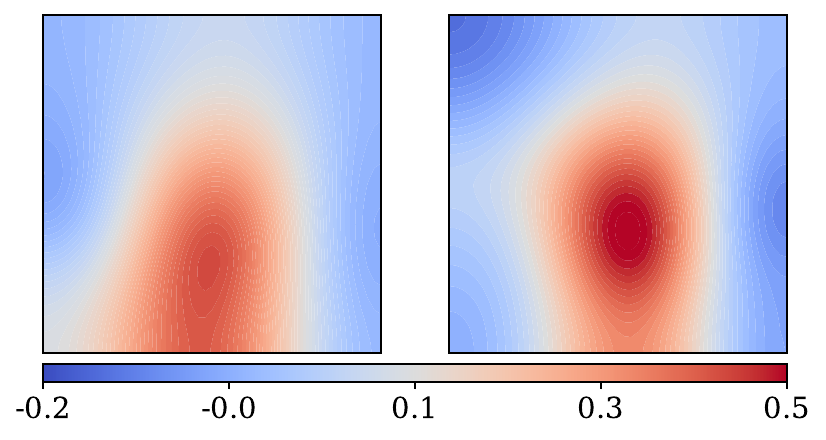}};
	\node[anchor=center, inner sep=0] (left) at (0, 0.382)
    {\includegraphics[width=0.4\textwidth]{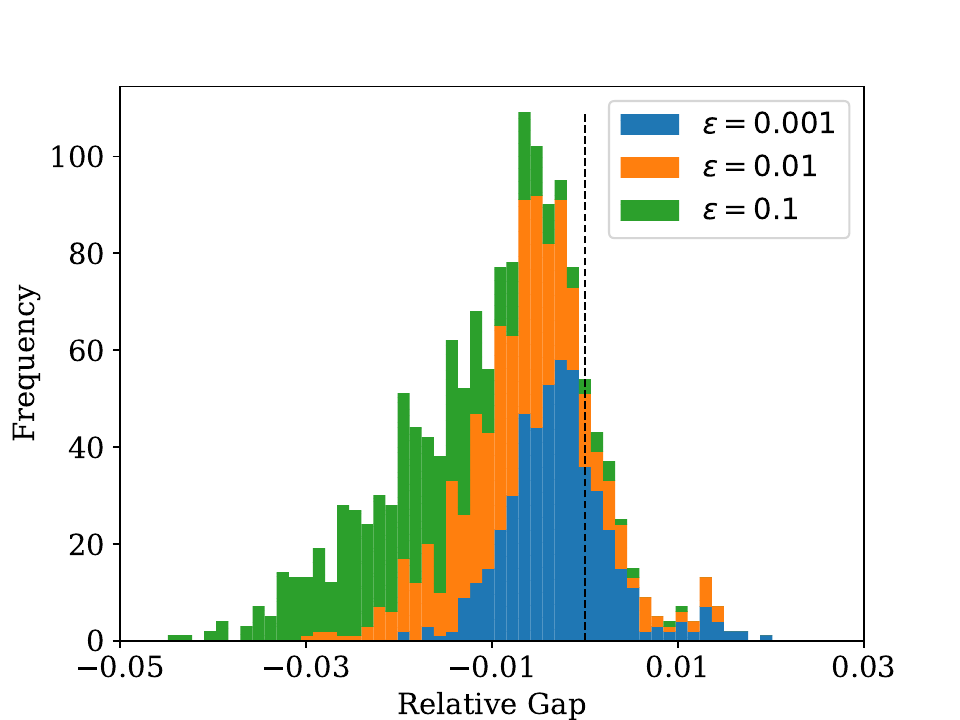}};
\end{tikzpicture}
\caption{Left: Histogram of the relative gap between solutions found by stochastic 
greedy and the deterministic greedy solution. 
Right: MAP estimates found using 
best-greedy (middle) and stochastic greedy with $\epsilon=10^{-3}$~(right).} 
\label{fig:linear-oed:hist-and-map}
\end{figure}

We next consider the computed sensor placements themselves in 
Figure~\ref{fig:inverse_heat_sensor_placements}.  The sensor placement obtained
using the best-greedy method is displayed in the left panel. The middle and
right panels summarize sensor placements obtained from 500 runs of the stochastic
cost-benefit greedy method with approximation parameters of $\epsilon = 10^{-3}$
and $\epsilon = 10^{-2}$. In these figures the color intensity
indicates the frequency of selecting specific sensors.  A potential advantage of
the stochastic cost-benefit greedy procedure is its ability to explore a broader
set of sensor placements.  The more compelling reason to use the stochastic
approach is its superior computational performance over the deterministic
best-greedy approach while consistently producing good quality sensor
placements. This is studied next. 

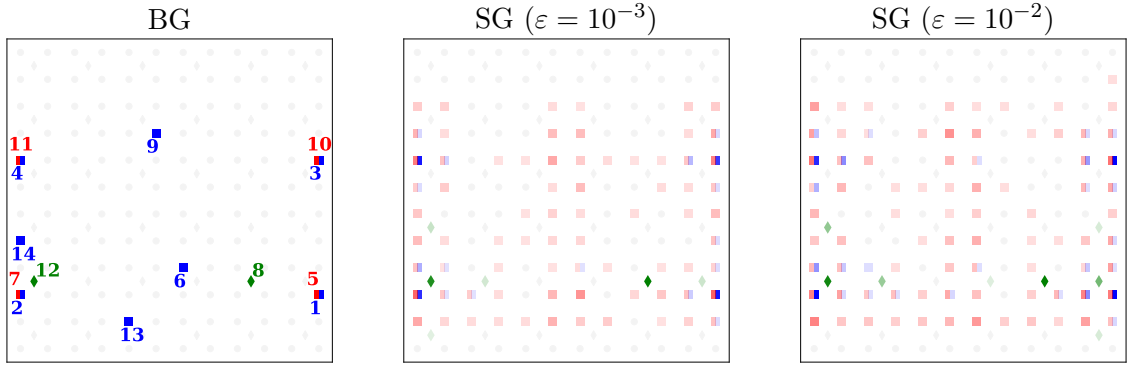
\begin{figure}[t!]
\centering
\begin{tikzpicture}
	\node[anchor=center, inner sep=0] (left) at (0, 0)
    {\includegraphics[width=0.28\textwidth]{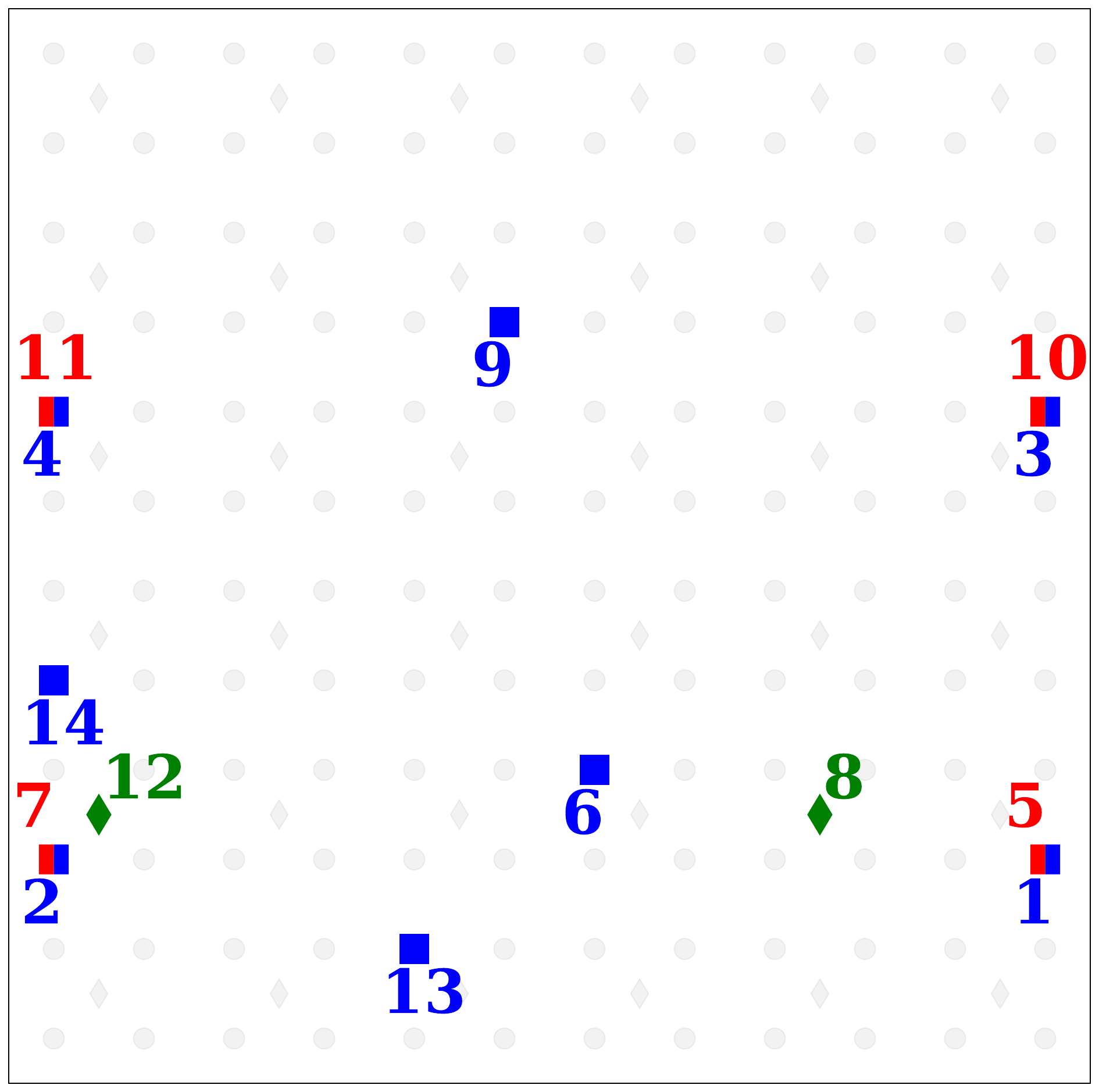}};
	\node[anchor=center, inner sep=0] (right) at (5.25, 0)
    {\includegraphics[width=0.28\textwidth]{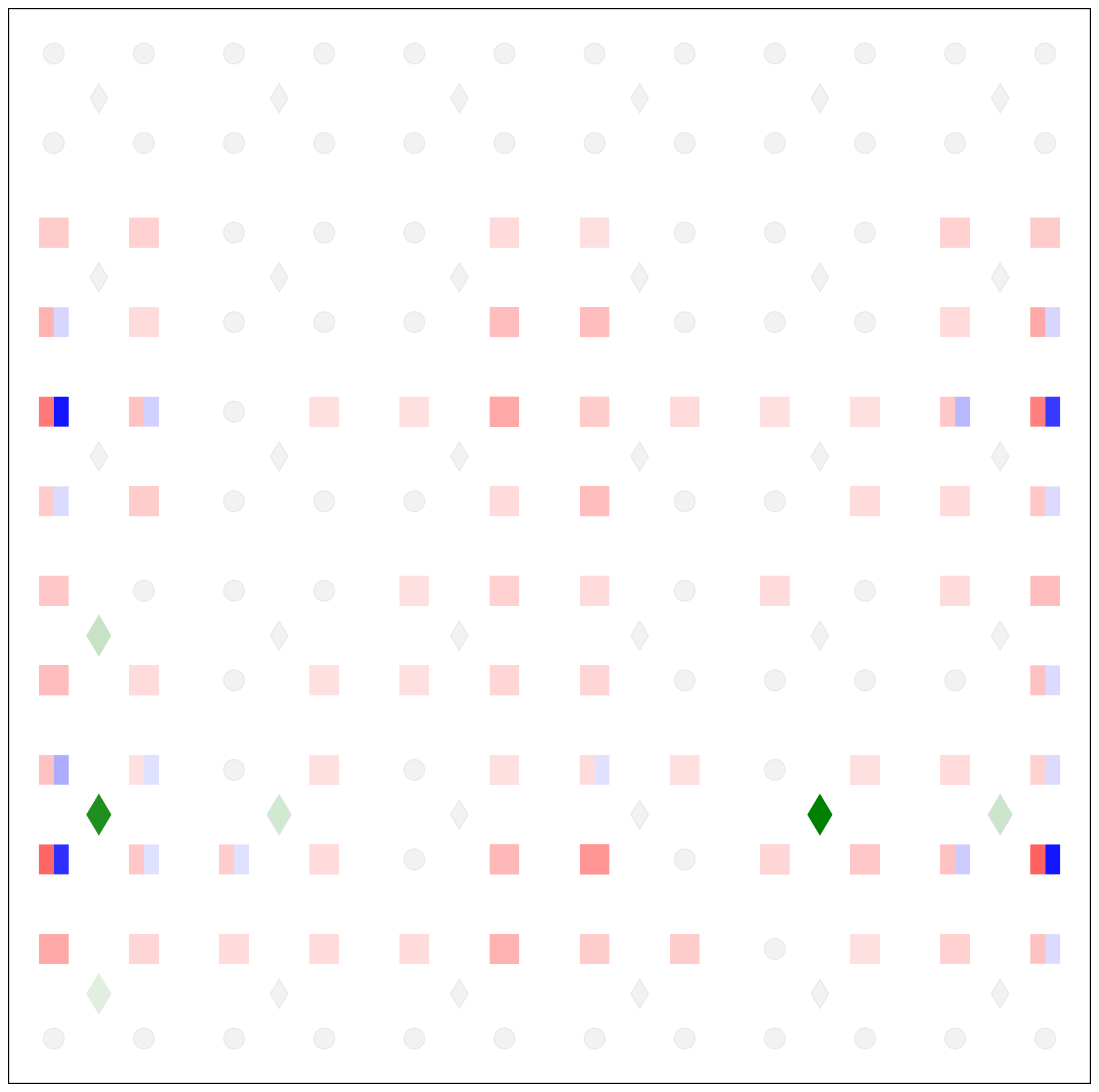}};
	\node[anchor=center, inner sep=0] (right) at (10.5, 0)
    {\includegraphics[width=0.28\textwidth]{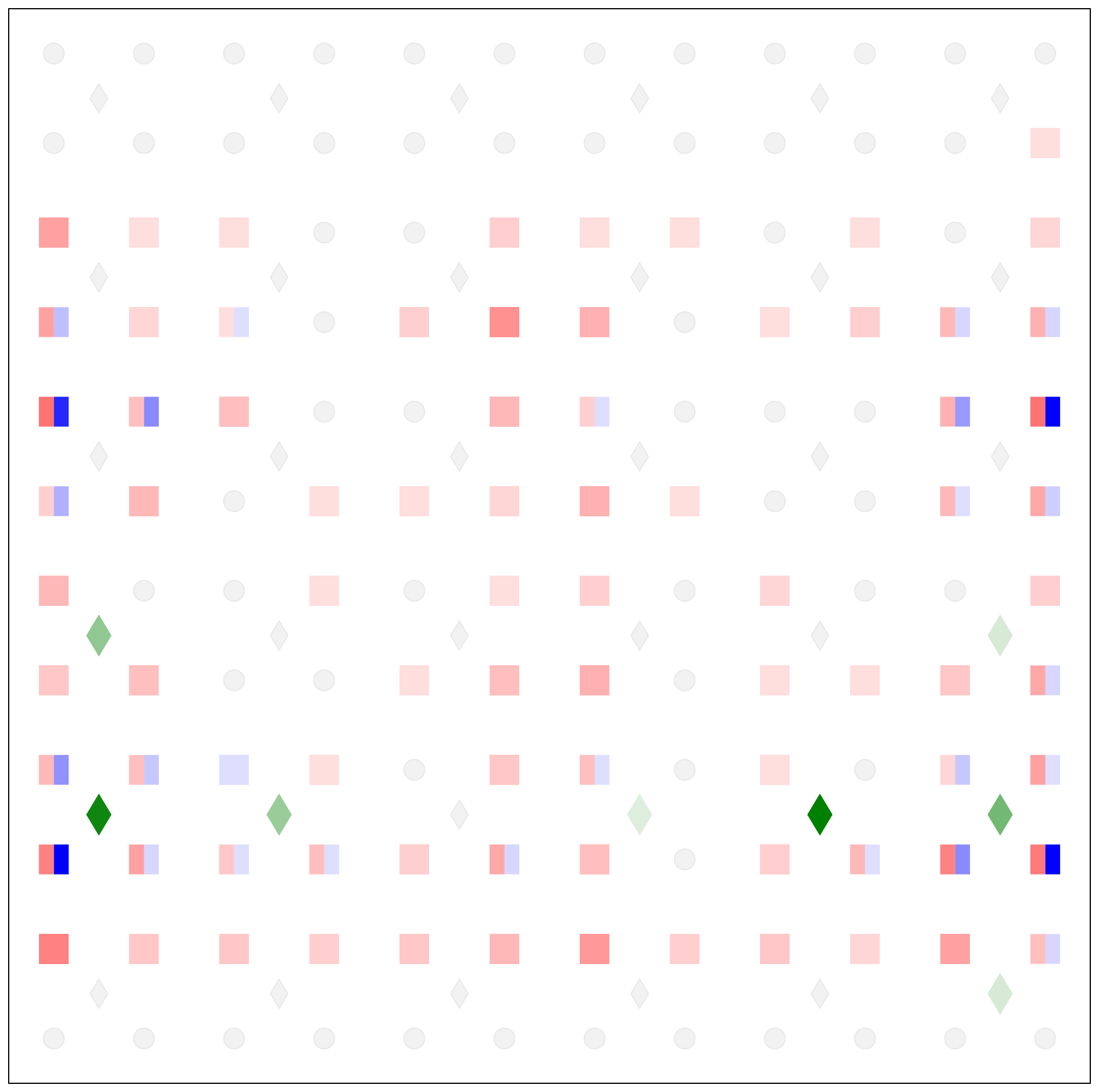}};
    \begin{scope}[anchor=center,text=black,align=center,pos=0.5]
        \node at (0, 2.4) {BG};
        \node at (5.25, 2.4) {SG ($\epsilon=10^{-3}$)};
        \node at (10.5, 2.4) {SG ($\epsilon=10^{-2}$)};
    \end{scope}
    \end{tikzpicture}
    \caption{
        Sensor placements found using Algorithms~\ref{alg:greedy-knapsack} and~\ref{alg:stochastic-cb-greedy}.
        The red (blue) squares correspond to a high (low) quality point evaluation sensor placement. 
        The green diamonds correspond to a point average sensor placement. 
        Left: sensor placements found by Algorithm~\ref{alg:greedy-knapsack}.
    Middle and right: sensors frequently chosen by Algorithm~\ref{alg:stochastic-cb-greedy} with $\epsilon = 0.01$ and $\epsilon=0.001$, respectively. The marker opacity indicates the frequency.}
    \label{fig:inverse_heat_sensor_placements}
\end{figure}

In Figure~\ref{fig:inverse-heat-deterministic-vs-stochastic}, 
we 
compare the performance of best-greedy and stochastic cost-benefit greedy methods 
over varying budget values. For each budget $B$ and $\epsilon$, we performed 500 trials.
The effect of $\epsilon$ is apparent:
smaller values close the performance gap between the stochastic and greedy
algorithm.  We additionally see the computational advantages of
Algorithm~\ref{alg:stochastic-cb-greedy}.  The computational cost of the
best-greedy method increases as the budget increases, while the cost of the
stochastic algorithm is asymptotically independent of the budget.  The initial
higher cost for smaller budget values is due to the presence of $B$ in the
denominator of~\eqref{eq:T-definition}. 

\begin{figure}[t!]
\centering
\begin{tikzpicture}
	\node[anchor=center, inner sep=0] (left) at (0, 0)
    {\includegraphics[width=0.4\textwidth]{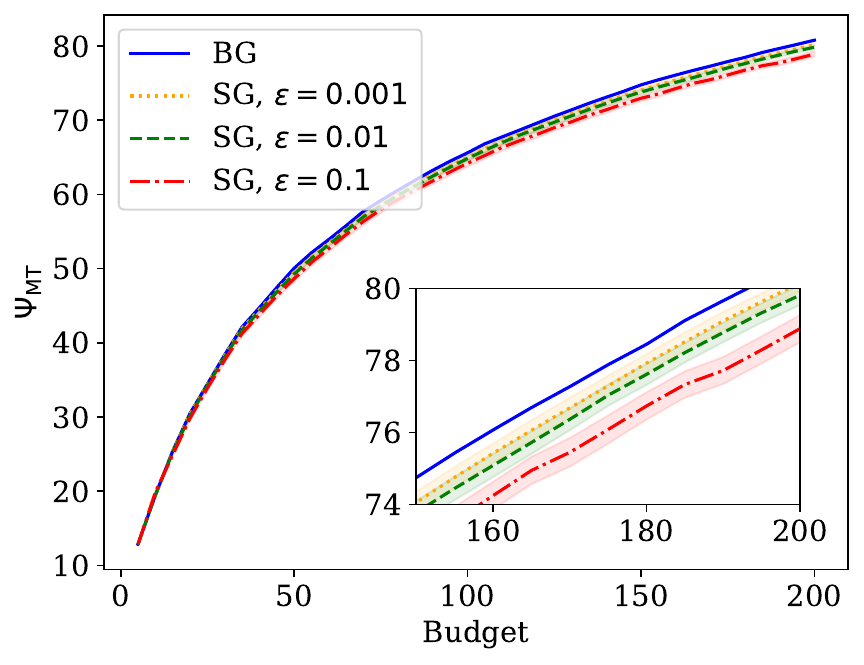}};
	\node[anchor=center, inner sep=0] (left) at (7.5, 0)
    {\includegraphics[width=0.425\textwidth]{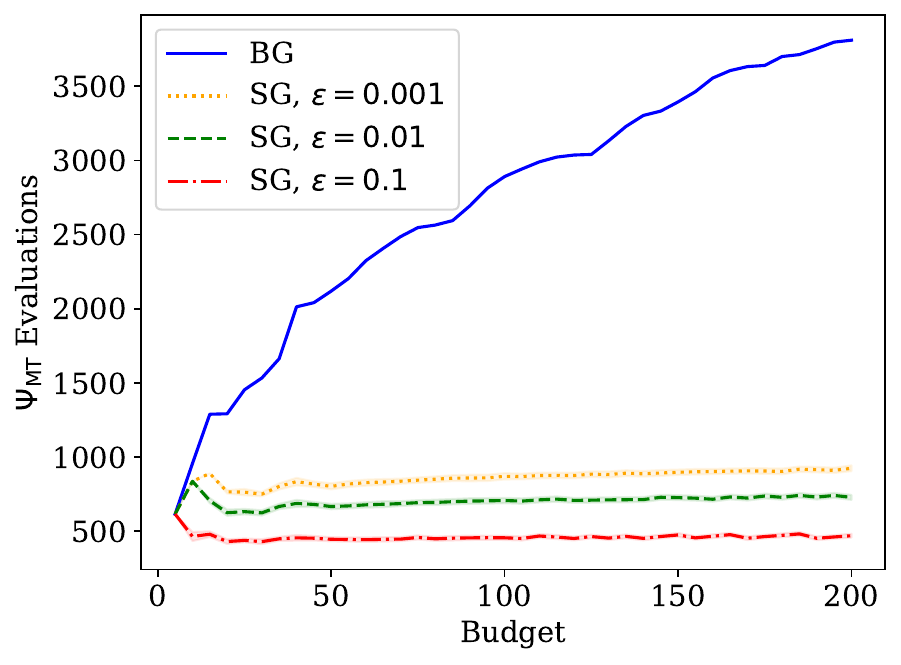}};
    \begin{scope}[anchor=center,text=black,align=center,pos=0.5,font=\bfseries]
        \footnotesize
        \node at (0.1, 2.85) {Performance};
        \node at (7.6, 2.85) {Computational Cost};
    \end{scope}
\end{tikzpicture}
\caption{Performance comparison of the deterministic algorithm against the stochastic algorithm. 
The solid blue line shows the performance of the deterministic greedy algorithm. 
The dotted, dashed and dash-dotted lines show the performance of the stochastic greedy algorithm 
with $\epsilon = 0.001, 0.01, 0.1$. The plots for the stochastic algorithms are the mean objective values.
The shaded regions around the curves are within two standard deviation of the mean.
}
\label{fig:inverse-heat-deterministic-vs-stochastic}
\end{figure}

\subsection{Subsurface Permeability Inference Using Two Sensor Types}\label{subsec:subsurface-permeability-inference}
To illustrate our approach for nonlinear Bayesian inverse problems, we consider
a model inverse problem motivated by groundwater quality monitoring
applications.  Specifically, we consider inference of a log-permeability field
in heterogeneous medium using measurements of hydraulic head pressure or
solute concentration at a few observation wells. 

\subsubsection{Problem setup}\label{subsubsec:subsurface_setup}
We first describe the domain $\Omega$ of the problem.
We let $\Omega = (0, 1)^2$ and denote its 
left, right, top and bottom boundaries by $\Gamma_L \coloneq \{0\}
\times [0,1], \Gamma_R \coloneq \{1\} \times [0,1], \Gamma_T \coloneq [0,1]
\times \{1\}, \Gamma_B = [0,1] \times \{0\}$, respectively.
We consider the relationship between the pressure $p$ and the log-permeability field
$\paramsFunc$ as described by the following elliptic PDE:
\begin{equation}
\begin{split}
-\nabla \cdot \left(e^{\paramsFunc} \nabla p \right) &= 0  \quad \text{ in } \Omega, \\
e^{\paramsFunc} \nabla p \cdot \bvec{n} &= 0  \quad \text{ on } \Gamma_T \cup \Gamma_B, \\
p &= 1 \quad \text{ on } \Gamma_L, \\
p &= 0 \quad \text{ on } \Gamma_R.
\end{split}
\label{eq:PDE_elliptic}
\end{equation} 

The transport of the solute through the medium depends on the pressure and
permeability field. We assume the solute is naturally present in the aquifer and
reaches a steady state at the time of measurement. The solute is advected by the
Darcy velocity field $\bvec{v}(\paramsFunc) = -e^{\paramsFunc(\bvec{x})}\nabla p$,
and diffused according to a known constant diffusivity coefficient of $D = 0.001$. The
steady-state concentration $c$ is thus governed by the steady 
advection-diffusion PDE: 
\begin{equation}\label{eq:PDE_ad-diff}
\begin{aligned}
-\nabla(D \nabla c) + \bvec{v}(\paramsFunc) \cdot \nabla c &= 0  \quad &&\text{ in } \Omega, \\
\left(\bvec{v}(\paramsFunc)c - D \nabla c\right)\cdot \bvec{n} &= 0  \quad &&\text{ on } \Gamma_T \cup \Gamma_B, \\
- D \nabla c \cdot \bvec{n} &= 0 \quad &&\text{ on } \Gamma_R, \\
c &= c_{\mathrm{in}}(y) \quad &&\text{ on } \Gamma_L .
\end{aligned}
\end{equation}
Here, solute enters the domain at the left boundary $\Gamma_L$ with inflow concentration $c_{\mathrm{in}}(y) = \exp(-\frac{(y-0.7)^2}{2(0.1)^2})$. A zero net
flux boundary condition is imposed on the top and bottom boundaries, $\Gamma_T
\cup \Gamma_B$, so that these boundaries are impermeable. A homogeneous
Neumann condition is imposed on the right boundary $\Gamma_R$, which allows the
solute to exit $\Omega$ via advective transport.   

To formulate the sensor selection problem, we fix a candidate set of
locations in the domain, shown in Figure~\ref{fig:nonlin_setup}, and assume
measurements are taken at a subset of these locations. Borehole drilling costs depend on factors such as depth, diameter, rock type, and casing
type~\cite{Everett1976,EPA1997}.  Typically, boreholes used for hydraulic head measurements require smaller diameters and less expensive casing than
those installed for measuring solute concentrations. Thus, the total
cost of concentration measurements, including borehole drilling and subsequent
laboratory testing, is generally higher than that of pressure measurements. In
our setup, the cost of a concentration measurement is taken to be 1.5 times that
of a pressure measurement. 

For practical simplicity, we assume that at most one sensor type is
permitted at each candidate location.  This assumption introduces packing constraints on
the feasible designs, and the resulting optimization problem does not fit
directly into the framework in Section~\ref{subsec:stoch_greedy}.  We
note, however, that greedy approaches utilizing continuous extensions have been
introduced for the more general case of multiple knapsack
constraints~\cite{BadanidiyuruVondrak2014,KulikShachnaiTamir2009}.  In
practice, the only modification to the deterministic and stochastic greedy algorithms 
is that when a sensor is selected at a given
location, both sensor types at that location are removed from the candidate
set. 

Measurement noise for both types is modeled as
additive and Gaussian with zero mean and standard deviations $\sigma_p =
\sigma_c = 0.05$ for pressure and concentration, respectively.  This corresponds
to approximately $5 \%$ of the $L_{\infty}$ norm of each state.

The log-permeability field $\paramsFunc$ is discretized on a triangular mesh
with $8192$ elements using piecewise linear basis functions, resulting in a
coefficient vector $\bvec{\params} \in R^{4225}$.  We assume a Gaussian prior 
$\mathcal{N}(\mpr,\Cmm)$. We take $\mpr
\equiv 0$, and define the prior covariance matrix $\Cmm$ as the finite element
discretization of the inverse squared elliptic operator $\mathcal{A}^{-2}$,
where $\mathcal{A} = \gamma \Delta + \delta I$, with $\gamma = 0.1, \delta =
0.5$. 
To mitigate undesirable boundary effects that can arise do the use of
the PDE operator, we impose Robin boundary conditions as described
in~\cite{RoininenHuttunenLasanen2014}. A sample from this prior, as well as
the corresponding pressure and concentration fields
solving~\eqref{eq:PDE_elliptic} and~\eqref{eq:PDE_ad-diff} are visualized
in Figure~\ref{fig:nonlin_setup}.

The elliptic forward problem~\eqref{eq:PDE_elliptic} is solved using a mixed
formulation, employing first-order Brezzi--Douglas--Fortin--Marini elements for
the velocity field $\bvec{v}$ and piecewise constant elements for the pressure
$p$~\cite{BrezziDouglasMarini1985}. The steady-state transport
problem~\eqref{eq:PDE_ad-diff} is solved using a streamline upwind
Petrov--Galerkin discretization in space to stabilize the advective
term~\cite{ElmanSilvesterWathen2014}. Both forward problems are implemented
in FEniCS~\cite{LoggMardalWells2012}. 

\begin{figure}[t]
\centering
    {\includegraphics[width=\textwidth]{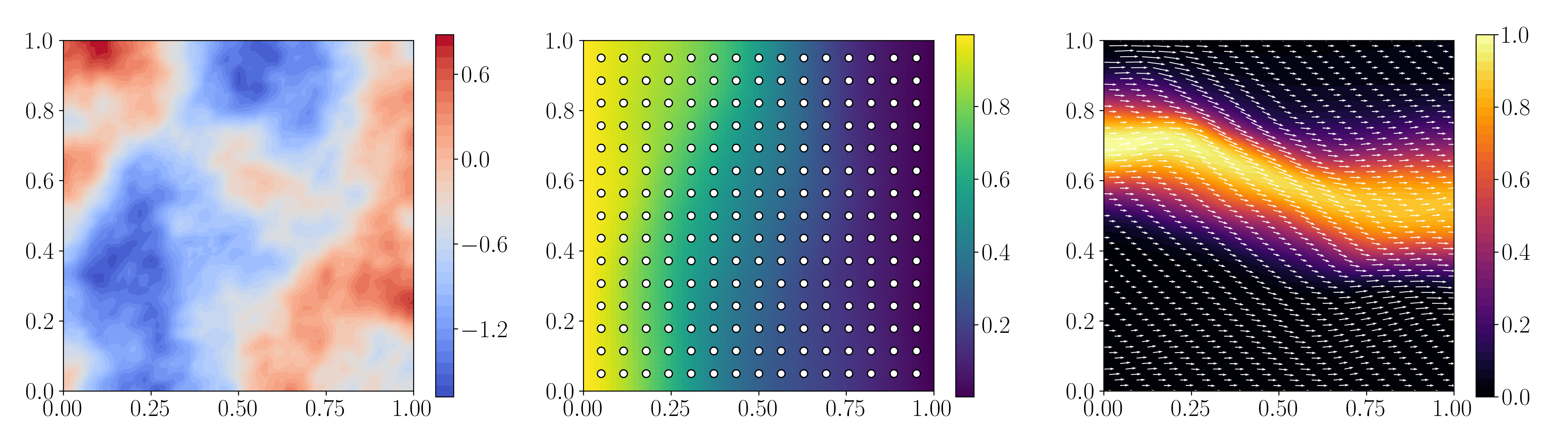}}
\caption{A sample log-permeability field (left) with the corresponding pressure (middle) and concentration (right). In the middle panel, white dots mark the candidate sensor locations.}
\label{fig:nonlin_setup}
\end{figure}

\subsubsection{Computational results}\label{subsubsec:subsurface_results}
We fix a total budget of $B = 15$ and compute the optimal designs using the
stochastic cost-benefit greedy approach outlined in Algorithm~\ref{alg:stochastic-cb-greedy}.  For the
present nonlinear inverse problem, we follow the strategy of using the \fo{}
global linear model, as detailed in Section~\ref{sec:nonlinOED}.  We use $N =
5{,}000$ samples to compute the first and second order statistics required by
the BAE approach (see Section~\ref{subsec:nonlin_computations}), as this was
found to produce sufficiently stable estimates empirically.  

We computed sensor placements using deterministic best-greedy approach
(Algorithm~\ref{alg:greedy-knapsack}) and the stochastic greedy method
(Algorithm~\ref{alg:stochastic-cb-greedy}).  For the latter, we ran the
algorithm with several choices of the approximation parameter $\epsilon$, 
with $50$ independent runs per $\eps$; see 
Figure~\ref{fig:transport_designComp}. We observe  clearly preferred sensor
regions for concentration and pressure measurements, which become
increasingly localized as $\epsilon$ decreases.  
In this problem, averaged over 50 runs, the stochastic-greedy runs
were approximately $12.6, 7,$ and $4.5$ times faster than the best
greedy approach for the tested values of $\epsilon = 0.1,0.01,0.001$,
respectively.
\begin{figure}[ht!]
\centering
\begin{tikzpicture}
	\node[anchor=south west, inner sep=0] (left) at (-0.5,0)
    {\includegraphics[width=\textwidth]{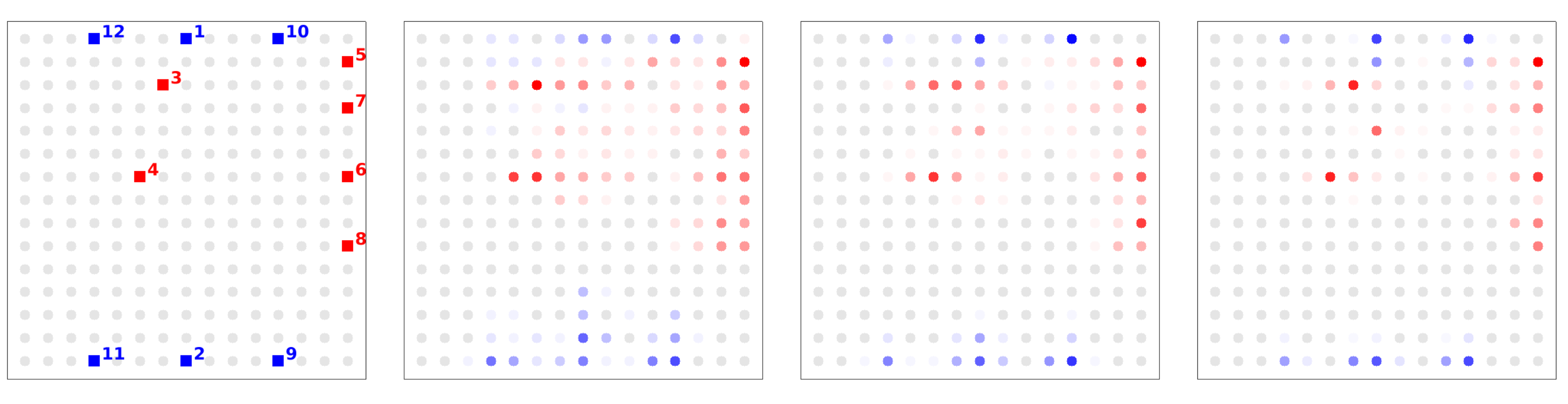}};
    \node at (1.4,4)  {{BG}};
    \node at (5.1,4)  {{SG} ($\epsilon = 10^{-1}$)};
    \node at (9.3,4)  {{SG} ($\epsilon = 10^{-2}$)};
    \node at (13.2,4) {{SG} ($\epsilon = 10^{-3}$)};
\end{tikzpicture}
\caption{Visualizations of the optimal designs for the coupled transport
problem. In the left figure, the blue and red squares mark the optimal pressure
and concentration sensors selected using the knapsack greedy
algorithm~\ref{alg:greedy-knapsack}. In the remaining columns, sensor placements
found using stochastic greedy with $\epsilon \in \{10^{-1},10^{-2},10^{-3}\}$ are visualized
for $50$ runs. The opacity of the circles indicate how frequently each sensor is
selected by the stochastic greedy method.}
\label{fig:transport_designComp}
\end{figure}
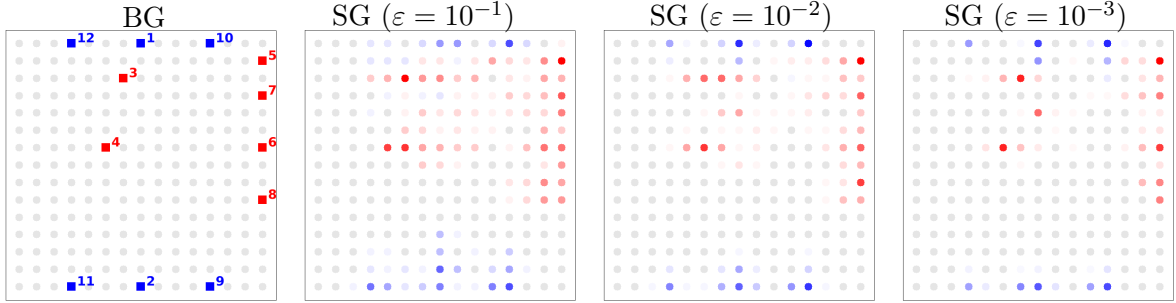

We next perform an empirical assessment of the quality of the optimal designs obtained 
using stochastic greedy with the \fo{} EIG as proxy for the 
exact EIG. This is particularly important in the present nonlinear
setting. Specifically, as noted in Section~\ref{subsec:nonlin_computations}, 
$\mfEIG^{\Op}$ 
is not
submodular, and thus the theoretical guarantees of the
deterministic and stochastic greedy algorithms no longer hold. 
We compute optimal sensor placements using 
\begin{itemize}[leftmargin=1em]
\item Algorithm~\ref{alg:stochastic-cb-greedy} with $\mfEIG^{\Op}$
in~\eqref{eq:error_corrected_EIG} as objective function (the proposed approach); 

\item Algorithm~\ref{alg:greedy-knapsack} with $\mfEIG^{\Op}$ as the objective function; and 
\item Algorithm~\ref{alg:greedy-knapsack} with a nested Monte Carlo (NMC) estimator for the 
EIG as the objective. 
\end{itemize}
The quality of the computed optimal designs are compared with $100$
randomly selected designs from the candidate set.  To provide an
unbiased assessment, for each of the
computed optimal designs and the random ones, we evaluate the EIG using an NMC estimator. 

Before discussing the numerical results, we briefly describe our approach for computing 
EIG via sampling. 
The standard NMC estimator of the EIG is
\begin{equation}
\EIG_{\mathrm{NMC}}(\designSubset) = \frac{1}{N_{\mathrm{out}}}\sum_{i=1}^{N_{\mathrm{out}}}\Big[ \log \likelihood(\data^{(i)}\vert \bvec{\params}^{(i)}; \designSubset) - \log \Big( \frac{1}{N_{\mathrm{in}}}\sum_{j=1}^{{N_{\mathrm{in}}}}\log \likelihood(\data^{(i)}\vert \bvec{\params}^{(i,j)}; \designSubset) \Big) \Big],
\label{eq:NMC}
\end{equation}
where $\{ \bvec{\params}^{(i)}, \data^{(i)}\}_{i=1}^{N_{\mathrm{out}}} \sim
\pi_{\params,\data}$, and $\{ \bvec{\params}^{(i,j)}\}_{j=1}^{N_{\mathrm{in}}}
\sim \pi_{\params}$ for $i =1 \ldots, N_{\mathrm{out}}$.  In our setting, model
evaluations are expensive, and the noise level is relatively
small, leading to highly concentrated likelihoods. As a result, obtaining
reliable estimates using standard NMC is prohibitively costly. 
We use the tractable option 
suggested in~\cite{HuanMarzouk2013}
where we reuse the outer samples to estimate the
evidence.
To prevent the self-likelihood term $\likelihood(\data^{(i)} \vert \params^{(i)}; \designSubset)$ from dominating the sum under concentrated likelihoods, we exclude it via a leave-one-out estimator~\cite{PooleOzairVanDenOordEtAl2019}:
\begin{align*}
\pi_{\data}(\data^{(i)}) \approx \frac{1}{N_{\mathrm{out}}-1}\sum_{j=1, j\neq i}^{{N_{\mathrm{out}}}} \likelihood(\data^{(i)}\vert \bvec{\params}^{(j)}; \designSubset).
\end{align*}
While this variation of NMC does not ensure asymptotically unbiased estimates
of the EIG, we found that it performed significantly better than standard NMC
within our sample budget.

A comparison of the optimal designs obtained using the greedy approaches and
randomly chosen designs is shown in Figure~\ref{fig:transport_designCompEIG}~(left).  Observe
that all greedy designs significantly outperform the random designs, indicating
that the greedy approaches perform well in practice.  For the deterministic
greedy algorithm combined with the NMC estimator, we used $50{,}000$ samples to
construct the design (distinct from those used for performance evaluation). 
Although the NMC-based approach yields slightly better designs, the marginal
performance gain comes at a significantly higher computational cost, requiring
10 times the number of PDE solves as the error-aware Gaussian approach described
in Section~\ref{subsec:nonlin_BAE}.  

Since the stochastic cost-benefit greedy algorithm produces different designs
across runs, we next compare its performance against the deterministic greedy
design statistically.  In Figure~\ref{fig:transport_designCompEIG}~(right), we
report a histogram of the relative performance gap between these designs. 
While the stochastic greedy approach generally underperforms deterministic greedy, 
the gap is typically small and it occasionally yields better designs.
This indicates that the stochastic
approach can produce competitive designs at a
significantly lower computational cost.

\begin{figure}[t!]
\centering
\begin{tikzpicture}
	\node[anchor=center, inner sep=0] (left) at (0, 0)
    {\includegraphics[width=0.45\textwidth]{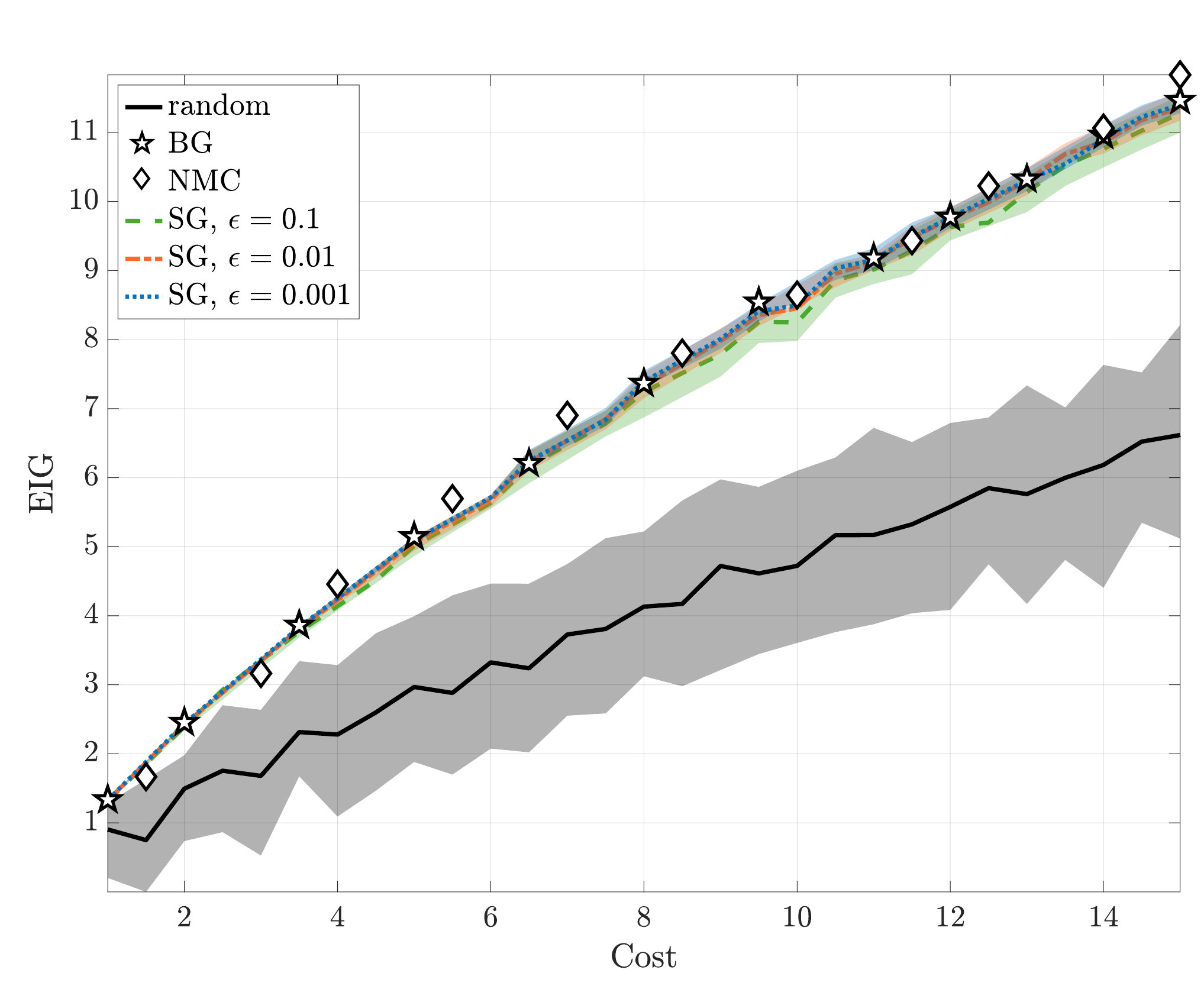}};
	\node[anchor=center, inner sep=0] (left) at (8, 0)
    {\includegraphics[width=0.45\textwidth]{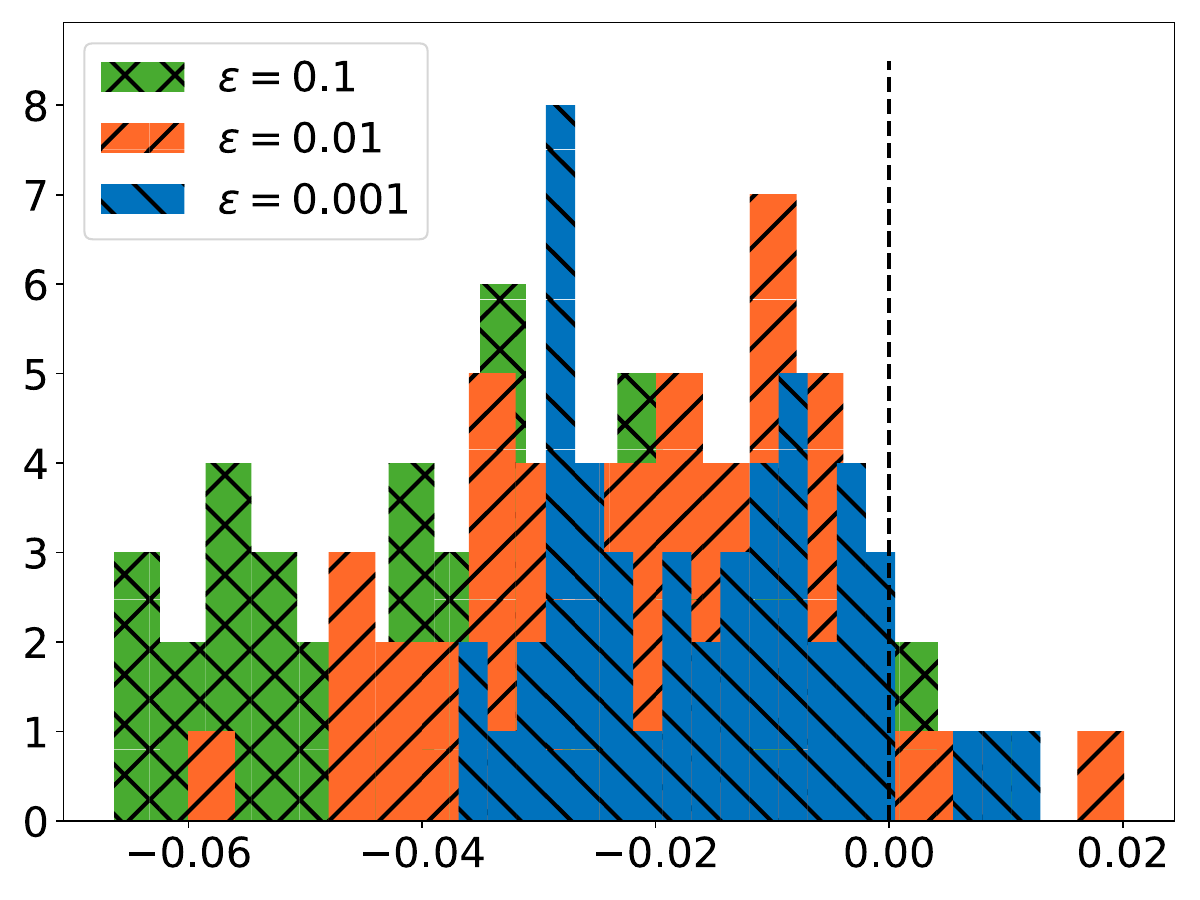}};
    \begin{scope}[anchor=center,text=black,align=center,pos=0.5,font=\bfseries]
    \end{scope}
\end{tikzpicture}
\caption{Left: EIG for designs selected
using Algorithm~\ref{alg:greedy-knapsack}, using both the NMC and Gaussian EIG estimators
(shown as white diamonds and stars, respectively). 
We also show the mean EIG and the $10-90$ percentile range using designs
obtained from 50 runs of Algorithm~\ref{alg:stochastic-cb-greedy} with three values of
$\varepsilon$ along with 100 randomly sampled designs. 
All EIG values are evaluated using an NMC estimator
with 50{,}000 samples. 
Right: Histogram of the relative gap between the deterministic greedy
solution and stochastic greedy solutions for different values of $\varepsilon$.}
\label{fig:transport_designCompEIG}
\end{figure}

\section{Conclusions}\label{sec:Conc}
We have addressed optimal placement of multi-type sensors for Bayesian linear
and nonlinear inverse problems governed by PDEs. The proposed framework provides
considerable flexibility: it allows sensors with different costs, fidelities,
and measurement types. Our theoretical results provide a solid foundation for
the proposed methods, and our numerical results indicate their effectiveness.

Our computational study of the model linear inverse problem in
Section~\ref{subsec:heat-source-inversion}, where the theoretical guarantees of
the stochastic greedy algorithm apply, illustrates the considerable advantages
of stochastic greedy over the deterministic best greedy algorithm. Our results
also provide insight into parameter inversion and data acquisition with multi-type
sensors. Further, we emphasize that the proposed stochastic cost-benefit greedy
method and its analysis are of independent interest for general submodular
optimization problems with knapsack constraints.

For nonlinear inverse problems, we use the \fo{} EIG as a proxy for the exact
EIG. Our theoretical results provide insight into the \fo{} global linear model
and establish that the \fo{} EIG is a lower bound on the exact EIG. The latter leads
to a principled formulation of the OED problem as a knapsack-constrained binary
optimization problem with \fo{} EIG as the objective. In that case, due to the
non-diagonal total error covariance matrix of the error-corrected data
likelihood, submodularity is lost and the theoretical guarantees of the greedy
methods do not apply. However, our numerical results in
Section~\ref{subsec:subsurface-permeability-inference} demonstrate the
effectiveness of the proposed stochastic greedy approach. At a significantly
reduced cost, the stochastic greedy method yields sensor placements that are
competitive with those obtained by optimizing a nested Monte Carlo estimate of
the exact EIG using the deterministic best greedy method. Furthermore,
sensor placements obtained using our proposed approach consistently outperform
random sensor placements. 

We also note that while the lower bound was established assuming a Gaussian
prior, the analysis can be extended to the non-Gaussian setting.  The error in
the EIG approximation gains an additional (non-positive) term due to the
non-Gaussianity of the prior.  While the \fo{} EIG is no longer guaranteed to
lower bound the exact EIG in that setting, this additional error term is
design-independent and thus does not directly affect the relative ranking of
sensor configurations.  
\bibliographystyle{abbrv}
\bibliography{refs_abbreviated}

\appendix

\begingroup
\setlength{\abovedisplayskip}{6pt plus 2pt minus 2pt}
\setlength{\belowdisplayskip}{6pt plus 2pt minus 2pt}
\setlength{\abovedisplayshortskip}{4pt plus 2pt minus 2pt}
\setlength{\belowdisplayshortskip}{4pt plus 2pt minus 2pt}

\section{Stochastic Cost-Benefit Greedy Auxiliary Results}\label{appendix:stochastic-cb-greedy}

The following results are used in the analysis of Algorithm~\ref{alg:stochastic-cb-greedy}.

\begin{lemma}[Intersection Probability]\label{lemma:intersection-probability}
    Suppose that the conditions of Lemma~\ref{lemma:cb-greedy-expected-iterative-gain} hold.
    Let $F$ denote the event that $R \cap (S^{*} \backslash S) \ne \emptyset$, where $S^{*}$ is a solution to~\eqref{eq:knapsack-submodular-max-problem}.
    Then,
    \begin{equation}
        \mathbb{P}(R \in F) \ge (1 - \epsilon )  \frac{c(S^{*} \backslash S)}{B}.
    \end{equation}
\end{lemma}
\begin{proof}
First, observe that
\begin{align*}
    \mathbb{P}(F^{c})
    &= \Big(1 - \frac{c(S^{*} \backslash S)}{c(V \backslash S)}\Big)^T
    \le \exp\Big(- \frac{T c(S^{*} \backslash S)}{c(V \backslash S)}\Big)
    \le \exp\Big(- \frac{T c(S^{*}\backslash S)}{c(V)}\Big).
\end{align*}
Using the fact that $1 - e^{-x}$ is concave,
and
letting $x_{1} = 0, x_{2} = T B / c(V)$, and $\lambda = c( S^{*}\backslash S) / B$, 
\begin{multline*}
    \mathbb{P}(F)
    = 1 - \mathbb{P}(F^{c})
\ge
1 - \exp\left( - \frac{T c(S^{*}\backslash S)}{c(V)}\right)
 \ge (1 - \lambda) (1 - e^{- x_{1}})\\
= \frac{c(S^{*} \backslash S)}{B} \left(1 - \exp\left(-\frac{T B}{c(V)}\right)\right)
= (1 - \epsilon) \frac{c(S^{*}\backslash S)}{B}.
\end{multline*}
\end{proof}

\begin{lemma}[Selecting a Remaining Optimal Sensor]\label{lemma:sample-remaining-element}
    Suppose the conditions of Lemma~\ref{lemma:cb-greedy-expected-iterative-gain} hold.
    Again, denote by $F$ the event that $R \cap (S^{*} \backslash S) \ne \emptyset$.
    Suppose we uniformly sample a sensor $a \in R \cap (S^{*}\backslash S)$, and denote this random variable by $A$.
    Then for $a \in S^{*}\backslash S$,
    \begin{equation*}
        \mathbb{P}\big\{A = a\,:\, R \cap (S^{*}\backslash S) \ne \emptyset\big\} = c(a) / c(S^{*} \backslash S).
    \end{equation*}
\end{lemma}

\begin{proof}
    Assume that $k := |R \cap (S^{*}\backslash S)|$. We will show that this does not depend on $k$.
    For $a \in S^{*}\backslash S$, the probability that $a$ appears $j \le k$ times in $R \cap (S^{*}\backslash S)$ is distributed according to a binomial distribution with parameters $k$ and $p = c(a) / c(S^{*}\backslash S)$.
    Thus, the probability that $a$ appears $j$ times is ${k \choose j} p^{j}(1-p)^{k-j}$.
    Also, if $a$ appears $j$ times then the probability of uniformly sampling it from $R \cap (S^{*}\backslash S)$ is $j / k$.
    By summing up over $j$ we get
    \begin{equation}
    \begin{aligned}
        \mathbb{P}\big\{A = a\,:\,|R\cap (S^{*}\backslash S)| = k\big\}
       &= \sum_{j=0}^{k} \frac{j}{k} {k \choose j} p^{j} (1 - p)^{k-j}\\
       &= \frac{1}{k} \sum_{j=0}^{k} j {k \choose j} p^{j} (1 - p)^{k-j}
       = \frac{1}{k} (kp) = p.
    \end{aligned}
    \end{equation}
    Because this does not depend on $k$, it then follows that $\mathbb{P}(A = a) = c(a) / c(S^{*}\backslash S)$.
\end{proof}

\end{document}